\documentclass[11pt]{amsart}

\usepackage[authoryear]{natbib}

\RequirePackage[OT1]{fontenc}
\RequirePackage{amsthm,amsmath,amssymb}
\usepackage[colorlinks=true, urlcolor=blue, linkcolor=black, citecolor=black, pdftex]{hyperref}
\usepackage{graphicx}
\usepackage[usenames]{color}
\usepackage{csquotes}
\usepackage[utf8]{inputenc}
\usepackage[authoryear]{natbib}

\RequirePackage[OT1]{fontenc}
\RequirePackage{amsthm,amsmath,amssymb}

\usepackage{graphicx}
\usepackage[usenames]{color}

\newtheorem{thm}{Theorem}[section]

\newtheorem{prop}{Proposition}[section]
\newtheorem{cor}{Corollary}[section]
\newtheorem{hyp}{Assumption}[section]
\newtheorem{definition}{Definition}[section]
\newtheorem{lemme}{Lemma}[section]
\newtheorem{Claim}{Claim}
\usepackage{memhfixc} 

\usepackage[utf8]{inputenc}
\usepackage{amsmath}
\usepackage{amsfonts}
\usepackage{amssymb}
\usepackage{amsthm}
\usepackage{enumerate}
\usepackage{mathrsfs} 
\usepackage{lscape}
\usepackage{bbm}

\renewcommand{\d}{\, \mathrm{d}}
\newcommand{\R}{\mathbb{R}}
\renewcommand{\S}{\mathbb{S}}

\renewcommand{\P}{\mathbb{P}\,}
\newcommand{\N}{\mathbb{N}}
\newcommand{\B}{\mathcal{B}}
\newcommand{\E}{\mathbb{E}}

\newcommand{\V}{\mathbb{V}}

\newcommand{\F}{\mathscr{F}}

\renewcommand{\L}{\mathbb{L}}
\newcommand{\M}{\mathcal{M}}

\newcommand{\I}{{\mathbf{I}}}

\newcommand{\W}{{\mathbb{W}}}

\newcommand{\Ker}{\text{Ker}\,}

\newcommand{\LL}{\mathbb{L}}

\newcommand{\XX}{\mathbb{X}}
\newcommand{\TT}{\mathbb{T}} 
\newcommand{\x}{\mathtt{x}}
\newcommand{\carre}{\qed}

\renewcommand{\t}{\mathtt{t}}

   \newcommand{\smallO}[1]{\ensuremath{\mathop{}\mathopen{}{\scriptstyle\mathcal{O}}\mathopen{}\left(#1\right)}}

\newcommand{\1}{\mathbbm{1}}
\newcommand{\ttheta}{\boldsymbol{\theta}}
\newcommand{\tttheta}{\boldsymbol{\hat{\theta}}}
\title{Model selection for Poisson processes with covariates}

\date{June, 2013} 
\author{Mathieu Sart} 
\address{Univ. Nice Sophia Antipolis, CNRS,  LJAD, UMR 7351, 06100 Nice, France.}
\email{msart@unice.fr}
\keywords{Adaptive estimation, Model selection, Poisson processes, T-estimator.} 
\subjclass[2010]{62M05, 62G05}

\begin{document}

\begin{abstract}
We observe $n$ inhomogeneous Poisson processes with covariates and aim at estimating their intensities.  We assume that the intensity of each Poisson process is of the form $s (\cdot, x)$ where $x$ is the covariate and where $s$ is an unknown function. We propose a model selection approach where the models are used to approximate the multivariate function $s$.
We show that our estimator satisfies an oracle-type inequality under very weak assumptions both on the intensities and the models.  By using an Hellinger-type loss, we establish non-asymptotic risk bounds and specify them under several kind of assumptions on the target function $s$ such as being smooth or a product function. Besides, we show that our estimation procedure is robust with respect to these assumptions.
\end{abstract}
\maketitle

\section{Introduction} \label{sectionintro}
We consider $n$ independent Poisson point processes $N_i$ for $i=1,\dots,n$ indexed by the measurable space $(\mathbb{T},\mathscr{T})$. For each $i$, we assume that the intensity of $N_i$ with respect to some reference measure~$\mu$ on $(\TT,\mathscr{T})$ is of the form $s_i (\cdot) = s(\cdot,x_i)$ where~$x_i$ is a deterministic element of some measurable set $(\XX,\mathscr{X})$ and~$s$ is a non-negative function on $\TT \times \XX$ satisfying 
\begin{eqnarray*} \label{conditionsurs}
\forall i \in \{1,\dots, n\}, \quad \int_{\TT} s(t,x_i) \d \mu(t) < + \infty.
\end{eqnarray*}
Typically, this corresponds to the modelling of the times of failure of $n$ repairable systems where the reliability of each of them depends on external factors measured by some covariates $x_1,\dots, x_n$, in which case $\TT$ corresponds to an interval of time, say $[0,1]$, and $\XX$ to some compact subset of $\R^k$, say $[0,1]^k$.
Our aim is to estimate~$s$ from the observations of the pairs $(N_i,x_i)_{1 \leq i \leq n}$. 

Let $\L_{+}^1 (\TT \times \XX, M)$ be the cone of integrable and non negative functions on $(\TT \times \XX, \mathscr{T} \otimes \mathscr{X} )$ equipped with the product measure $M =  \mu \otimes \nu_n$ where  $\nu_n =  n^{-1} \sum_{i=1}^n \delta_{x_i}$. In order to evaluate the risks of our estimators, we endow $\L_{+}^1 (\TT \times \XX, M)$ with the Hellinger-type distance $H$ defined for $u,v \in \L_{+}^1 (\TT \times \XX, M)$  by 
\begin{eqnarray*}
2 H^2(u,v) &=&  \int_{\TT \times \XX}  \left( \sqrt{u (t,x)} - \sqrt{v(t,x)}  \right)^2 \d M(t,x)  \\
&=& \frac{1}{n} \sum_{i=1}^n \int_{\TT} \left( \sqrt{u (t,x_i)} - \sqrt{v(t,x_i)}  \right)^2 \d \mu (t).
\end{eqnarray*}
Let $(\L^2 (\TT \times \XX, M), d_2)$ be the metric space of  functions $f$ on $\TT \times \XX$ such that $f^2$  belongs to $\L_{+}^1 (\TT \times \XX, M)$. Given  a suitable collection $\V$ of models (i.e subsets of $\L^2 (\TT \times \XX, M)$  which are not necessarily linear spaces) and a non-negative application~$\Delta$ on $\V$   satisfying
$$\sum_{V \in \V} e^{-\Delta(V)} \leq 1 ,$$
we build an estimator $\hat{s}$ whose risk $\E \left[H^2(s,\hat{s}) \right]$ satisfies
\begin{eqnarray}\label{premiereineqoracle}
C \E \left[ H^2(s,\hat{s}) \right] \leq  \inf_{V \in \V} \left\{d_2^2 \left( \sqrt{s}, V \right) + \eta_V^2 + \frac{\Delta(V)}{n} \right\},
\end{eqnarray}
where $C$ is an universal positive constant, $d_2 \left( \sqrt{s}, V \right) $ is the  $\LL^2$-distance between $\sqrt{s}$ and  $V$ and  $n \eta_V^2$ is the metric dimension (in a suitable sense) of $V$.  We shall use this inequality in order to derive  risk bounds for our estimator under smoothness or structural assumptions on the target function $s$.

 In the literature, much attention has been paid to the problem of estimating the intensity of a Poisson process without covariates. Concerning estimation by model selection, \cite{Reynaud-Bouret2003} dealt with the $\LL^2$-loss, and provided a model selection theorem for a family of linear spaces~$V$. \cite{BaraudBirgeHistogramme} used the Hellinger distance and considered the case where the sets $V$ consist of piecewise constants functions on a partition of $\TT$. More general models were considered by~\cite{BirgePoisson} allowing for $V$ any subset  with finite metric dimensions (in a suitable sense).

Our statistical setting includes that of  Poisson regression. Indeed, if one observes $n$ independent random variables $Y_1,\dots,Y_n$, such that $Y_i$ obeys to a Poisson law with parameter~$f (x_i)$, one can estimate~$f$ by setting $\TT = \{0\}$, $\mu = \delta_0$ the Dirac measure on $\TT$ and $N_i (\{0\}) = Y_i$. In this case, $s(0,\cdot) = f (\cdot)$ and estimating~$s$ amounts to estimating $f$.
This last issue has been studied in~\cite{Antoniadis2001},~\cite{Antoniadis2001a},~\cite{BaraudMesure} and~\cite{Krishnamurthy2010} among other references. 
For the particular cases of Poisson regression and estimating the intensity of a single Poisson process, our results recover those of~\cite{BaraudMesure}.  

If we except these cases, statistical procedures that can estimate $s$ from $n$ independent Poisson processes with covariates are rather scarce. The only risk bounds we are aware of are due to~\cite{Comte2008} who considered  the $\LL^2$-loss and penalized projection estimators on linear spaces. Their approach requires that the intensity~$s$ be bounded from above by a quantity that needs to be either known or suitably estimated. Besides, they impose some restrictions on the  family of linear spaces $V$ in order that their estimator possesses minimax properties over classes of functions which are smooth enough.
 
Our approach is based on robust testing. We  propose a test inspired from  a variational formula in~\cite{BaraudMesure} and then apply the general methodology for model selection developed in~\cite{BirgeTEstimateurs}. This yields a $T$-estimator $\hat{s}$ that possesses nice (adaptation and  robustness) properties but suffers from the fact that its construction is numerically intractable. 
This estimator should be thus considered as a benchmark for what theoretical feasible.
We obtain an oracle inequality of the form~{(\ref{premiereineqoracle})} under very mild assumptions both on the intensity $s$ and the family of models~$\V$. This allows to derive risk bounds over a large range of Hölderian spaces including irregular ones. 

We shall also consider functions $s$ defined on a subset $\TT \times \XX$ of a  linear space with large dimension, say $\TT \times \XX = [0,1]^{1+k}$ with a large value of $k$. It is well known that in such a situation, the minimax approach based on smoothness assumptions may lead to very slow rates of convergence. This phenomenon is known as {the curse of dimensionality}.
In this case, an alternative approach is to assume that $s$ belongs to classes  $\F$ of functions satisfying structural assumptions  (such as the multiple index model, the generalized additive model, the multiplicative model \dots) and for which faster rates of convergence can be achieved. Very recently, this approach was developed by~\cite{Juditsky2009} (for the Gaussian white noise model) and  by~\cite{BaraudComposite} (in more general settings).  Unlike~\cite{Juditsky2009} we shall not assume that $s$ belongs to $\F$ but rather consider $\F$ as an approximating class for $s$.

In the present paper, our point of view is closer to that developed in~\cite{BaraudComposite}.
We shall use our new model selection theorem in conjunction with suitable families $\V$ of models in order to design an estimator $\hat{s}$ possessing good statistical properties with respect to many classes of functions of interest, including classes $\F = \F_{\times}$ of  product functions $(t,x) \mapsto u(t) v(x)$.   When $s (t,x)$ is of the form (or close to) $u(t) v(x)$, where $u$ and $v$ are assumed to be smooth we shall prove that our estimator is fully adaptive with respect to the regularities of both $u$ and~$v$. We shall also consider structural assumptions on the functions $u$ and~$v$ as well as parametric ones when $t$ and $x$ lie in a large dimensional space. 
We shall study the situation where, the intensity of each Poisson process belongs to a parametric class of functions $\F_{\Theta}$ with $\Theta \subset \R^k$. This means  that there exists some element $f_{\theta (x_i)} \in  \F_{\Theta}$ such that $s(\cdot,x_i) = f_{\theta (x_i)} (\cdot)$, and our aim is then to estimate the mapping $x \mapsto \theta(x)$ by model selection.

This paper is organized as follows. The general model selection theorem can be found in Section~2.
In Section~3, we study the case where $\F$ is a class of  smooth functions, and in Section~4  the case where $\F$ is a class of product functions. The problem of estimating $s$ when the intensity of each Poisson process~$N_i$ belongs to the same parametric model is dealt in Section~5.
Section~6 is devoted to the proofs.

Let us introduce some notations that will be used all along the paper.  We set  $\N^{\star} = \N \setminus \{0\}$,   $\R^{\star} = \R \setminus \{0\}$.  The components of a vector $\ttheta \in \R^k$ are denoted by $\ttheta = (\theta_1,\dots,\theta_k)$.
The numbers $x \wedge  y$ and $x \vee y$ stand for $\min(x,y)$ and $\max(x,y)$ respectively. 
For $(E,\mathcal{E},\nu)$ a measured space, we  denote by $\L^2 (E, \nu)$ the linear space of measurable functions~$f$ such that $\int_{E} |f|^2 \d \nu < \infty$. When $(E,\nu) = (\TT \times \XX, M)$, the corresponding $\LL^2$-distance is denoted by $d_2$, and the norm by $\|\cdot\|_2$. Alternatively, this distance (respectively this norm) is denoted by $d_{\t}$ (respectively $\|\cdot\|_{\t}$) when $(E,\nu) = (\TT, \mu)$, and by $d_{\x}$ (respectively $\|\cdot\|_{\x}$) when $(E,\nu) = (\XX, \nu_n)$. The supremum norm of a bounded function $f$ on a domain $E$ is denoted by  $\|f\|_{\infty} = \sup_{x \in E} |f(x)|$.
For $(E,d)$ a metric space, $x \in E$ and $A \subset E$, the distance between $x$ and $A$ is denoted by $d(x,A)= \inf_{a \in A} d(x,a)$. The closed ball centered at $x \in E$ with radius $r$ is denoted by $\B (x,r)$. The cardinality of a finite set $A$ is denoted by $|A|$.
We use $\F$ as a generic notation for a family of functions of $\L^2 (\TT \times \XX, M)$ of special interest.  
The notations $C$,$C'$,$C''$\dots are for constants. The constants $C$,$C'$,$C''$\dots   may change from line to line.

\section{A general model selection theorem}
Throughout this paper, a model $V$ is a subset of~$\L^2 (\TT \times \XX, M)$ with bounded metric dimension, 
in the sense of Definition~6 of \cite{BirgeTEstimateurs}.  We recall this definition below.

\begin{definition} \label{definitionmetricspace}
Let $V$ be a subset of $\L^2 (\TT \times \XX, M)$ and $D_V$ a right-continuous map from $(0,+\infty)$ into $ [1/2, + \infty)$ such that $D_V(\eta) = \smallO{\eta^2}$ when $\eta \rightarrow + \infty$. We say that $V$ has a metric dimension bounded by $D_V$ if  for all $\eta > 0$, there exists $S_V (\eta) \subset \L^2 (\TT \times \XX, M)$ such that for all $f \in \L^2 (\TT \times \XX, M)$, there exists $g \in S_V(\eta)$ with $d_2(f,g) \leq \eta$ and such that
$$\forall \varphi \in \L^2 (\TT \times \XX, M), \, \forall x \geq 2 ,\quad  |S_V (\eta) \cap \B (\varphi , x \eta)| \leq \exp \left( D_V (\eta)  \, x^2 \right).$$
Moreover, if one can choose $D_V$ as a constant, we say that $V$ has a finite metric dimension bounded by $D_V$.
\end{definition}
 This notion is more general than the dimension for linear spaces since a  linear space $V$ with finite dimension (in the usual sense) has a finite metric dimension. Besides, if $V$ is not reduced to $\{0\}$ one can choose  $D_V = \dim V$, what we shall do along this paper. The link with the classical definition of metric entropy may be found in Section~6.4.3 of~\cite{BirgeTEstimateurs}.
 Other models of interest with  bounded metric dimension  will appear later in the paper.  
 
Given a collection of such subsets, our approach is based on model selection. We propose a selection rule based on robust testing in the spirit of the papers \cite{BirgeTEstimateurs, BaraudMesure}. The test and the selection rule which are mainly abstract are postponed to Section~\ref{sectionpeuves}. 	The main result is the following.

\begin{thm}  \label{thmselectionmodelegeneral}
Let $\V$ be an at most countable family of models $V$ with bounded metric dimension $D_V (\cdot) $ and $\Delta$ be a mapping from $\V$ into $[0,+\infty)$ such that  $$\sum_{V \in \V} e^{-\Delta(V)} \leq 1.$$
There exists an estimator~$\hat{s} \in \L^1_+ (\TT \times \XX, M)$  such that, for all $\xi > 0$,
\begin{eqnarray*}\label{eqnoraclegeneralnonintegree}
\P \left[C H^2(s, \hat{s})  \geq   \inf_{V \in \V} \left\lbrace d^2_2 \left(\sqrt{s},V \right) + \eta_V^2  + \frac{  \Delta (V) }{n} \right\rbrace + \xi \right] \leq e^{-n  \xi},     
\end{eqnarray*}
where $C$ is an universal positive constant and where 
\begin{eqnarray*}\label{eqnetaV}
\eta_V = \inf \left\{ \eta > 0 , \; \frac{D_V (\eta)}{\eta^2} \leq   n \right\}. 
\end{eqnarray*}
In particular, by integrating the above inequality,
\begin{eqnarray}\label{eqnoraclegeneral}
C' \E\left[H^2(s, \hat{s}) \right]  \leq   \inf_{V \in \V} \left\lbrace d^2_2 \left(\sqrt{s},V \right) + \eta_V^2  + \frac{  \Delta (V) }{n} \right\rbrace ,     
\end{eqnarray}
where $C'$ is an universal positive constant.
\end{thm}

The condition $\sum_{V \in \V} e^{-\Delta(V)} \leq 1$ can be interpreted as a (sub)probability on the collection~$\V$. The more complex the family $\V$, the larger the weights $\Delta(V)$.
When~$\V$ consists of linear spaces~$V$ of finite dimensions $D_V$  one can take $\eta_V^2 =  D_V /n$   and hence~{(\ref{eqnoraclegeneral})} leads to
\begin{eqnarray*}
C' \E\left[H^2(s, \hat{s}) \right]  \leq   \inf_{V \in \V} \left\lbrace  d^2_2 \left(\sqrt{s}, V \right)+  \frac{ D_V +  \Delta (V) }{n} \right\rbrace.    
\end{eqnarray*}
When one can choose $\Delta(V)$ of order $D_V$, which means that the family $\V$ of models does not contain too many models per dimension, the estimator $\hat{s}$ achieves the best trade-off  (up to a constant) between the approximation  and the variance terms.

In the remaining part of this paper, we shall consider subsets $\F \subset \LL^2 (\TT \times \XX,M)$  corresponding to various assumptions on  $\sqrt{s}$ (smoothness, structural, parametric assumptions \dots). For such an~$\F$, we  associate a collection $\V_{\F}$ and deduce from  Theorem~\ref{thmselectionmodelegeneral} a risk bound for the estimator~$\hat{s}$ whenever $\sqrt{s}$ belongs or is close to  $\F$. This bound takes the form 
\begin{eqnarray}\label{relationepsilonF}
C'' \E \left[ H^2(s,\hat{s}) \right] \leq \inf_{f \in \F} \left\{ d_2^2 (\sqrt{s}, f) + \varepsilon_{\F} (f) \right\}  
\end{eqnarray}
where $$ \varepsilon_{\F} (f) = \inf_{V \in \V_{\F}} \left\{d_2^2 (f, V) + \eta_V^2 + \frac{\Delta(V)}{n}\right\},$$
and we shall bound the term $\varepsilon_{\F} (f)$ from above.
This upper bound will mainly depend on some properties of $f$, for example smoothness ones. In this case,  this result  says that if $\sqrt{s}$ is irregular but sufficiently close to a smooth function $f$, the bound we  get essentially corresponds to the one we would get for $f$. This can be interpreted as a robustness property.

Sometimes, several assumptions on $\sqrt{s}$ are plausible, and one does not know what class $\F$ should be taken. A solution is to consider $\mathfrak{F}$ a collection of such classes $\F$ and to use the proposition below to get an estimator whose risk satisfies (up to a remaining term)  relation~{(\ref{relationepsilonF})} simultaneously for all classes  $\F \in \mathfrak{F}$.
\begin{prop} \label{propmixing}
Let $\mathfrak{F}$ be an at most countable collection of subsets of $\LL^2 (\TT \times \XX, M)$ and $\bar{\Delta}$ be a mapping on $\mathfrak{F}$ into $[0,+\infty)$ such that $\sum_{\F \in \mathfrak{F}} e^{- \bar{\Delta} (\F)} \leq 1$.
For all $\F \in \mathfrak{F}$, let $\V_{\F}$ be a collection of models and $\Delta_{\F}$ be a mapping such that the assumptions of Theorem~\ref{thmselectionmodelegeneral} hold. 

There exists an estimator $\hat{s}$ such that, for all $\F \in \mathfrak{F}$, 
\begin{eqnarray*} 
 C \E \left[ H^2(s,\hat{s}) \right] \leq \inf_{f \in \F} \left\{ d_2^2 (\sqrt{s}, f) + \varepsilon_{\F} (f) \right\} + \frac{\bar{\Delta}(\F)}{n} ,
 \end{eqnarray*}
 where $$\varepsilon_{\F} (f) = \inf_{V \in \V_{\F}} \left\{d_2^2 (f, V) + \eta_V^2 + \frac{\Delta_{\F} (V)}{n}\right\},$$
and where $C$ is an universal positive constant.
\end{prop}

\section{Smoothness assumptions} \label{sectionhypclassiques}
Let  $\mathbf{I} = \prod_{j=1}^k I_j$ where the $I_j$ are intervals of $\R$  and  $\boldsymbol{\alpha} = \boldsymbol{\beta} + \mathbf{p}  \in (0,+\infty)^k$ with  $\mathbf{p} \in \N^k$ and  $\boldsymbol{\beta} \in (0,1]^k$.  A function $f$ belongs to the Hölder class $\mathcal{H}^{\boldsymbol{\alpha}} (\I)$, if there exists $L(f) \in [0,+\infty)$ such that for all $(x_1,\dots,x_k) \in \I$ and all $j \in \{1,\dots,k\}$, the functions $f_j (x) = f (x_1,\dots,x_{j-1},x,x_{j+1},\dots,x_k)$ admit a derivative of order ${p}_j$ satisfying
$$ \left| f_j^{({p}_j)} (x) - f_j^{({p}_j)}(y) \right| \leq L(f) |x - y|^{{\beta}_j} \quad  \forall x,y \in I_j.$$
The  class $\mathcal{H}^{\boldsymbol{\alpha}} (\I)$ is said to be  isotropic when the $\alpha_j$ are all equal,  and   anisotropic otherwise, in which case $\bar{\boldsymbol{\alpha}}$ given by $\bar{\boldsymbol{\alpha}}^{-1} = k^{-1} \sum_{j=1}^k \alpha_j^{-1}$ corresponds to the average smoothness of a function~$f$ in $\mathcal{H}^{\boldsymbol{\alpha}} (\I)$. We define the class of  Hölderian functions on~$\I$ by
  $$\mathcal{H}  \left(\I \right) = \bigcup_{\boldsymbol{\alpha}  \in (0,+\infty)^k } \mathcal{H}^{\boldsymbol{\alpha}} \left(\I \right).$$
Assuming that $\sqrt{s}$ is Hölderian corresponds thus to the choice $\mathscr{F} = \mathcal{H}  \left(\TT \times \XX \right)$.
Anisotropic  classes of smoothness are of particular interest in our context since the function $s$ depends on variables $t$ and $x$ that may play very different roles.

Families of linear spaces possessing good approximation properties with respect to the elements of $\F$  can be found in the literature. We refer to the results of~\cite{Dahmen1980}.
We may use these linear spaces (models) to approximate the elements of $\F$, and deduce from Theorem~\ref{thmselectionmodelegeneral} the following result.

\begin{cor}\label{coorminimiaxrate2}
Let us assume that  $\TT \times \XX = [0,1]^{k}$ and that $\mu$ is the Lebesgue measure.
There exists an estimator $\hat{s}$ such that for all $f \in \mathcal{H}  \left( [0,1]^k \right)$,  
\begin{eqnarray} \label{relationminimaxrate}
C  \E\left[ H^2(s,\hat{s})  \right]  \leq   d_2^2 (\sqrt{s} , f )  +   L \left( f \right)^{\frac{2 k}{2 \bar{\boldsymbol{\alpha}}+k}}   n^{- \frac{2 \bar{\boldsymbol{\alpha}}}{2 \bar{\boldsymbol{\alpha}} + k}} + n^{-1}
\end{eqnarray}
where $\boldsymbol{\alpha} \in (0,+\infty)^k$ is such that $f \in \mathcal{H}^{\boldsymbol{\alpha}} ([0,1]^k)$ and where $C > 0$ depends only on $k$ and $\max_{1 \leq j \leq k} \alpha_j $.
\end{cor}
Remark that the risk bound given by inequality~{(\ref{relationminimaxrate})} holds without any restriction on~$\boldsymbol{\alpha}$. Such a generality can be obtained since our model selection theorem is valid for any collection $\V$ of finite dimensional linear spaces. Some restrictions on the dimensionality of the linear spaces $V \in \V$ (as in~\cite{Comte2008})  would prevent us to get this rate of convergence for the Hölder classes $\mathcal{H}^{\boldsymbol{\alpha}}  \left( [0,1]^k \right)$ when $\min_{1 \leq j \leq k} \alpha_j$ is too small.
 
The preceding risk bound is quite satisfactory if $k$ is small but becomes worse when $k$ increases. We shall therefore consider other types of classes in the next section in order to avoid this {curse of dimensionality}.

\section{Families $\F$ of product functions} \label{multiplicativetimeindepent}
A common way of modelling  the  influence of the covariates on the number of failures of $n$ systems is to assume that, for each  $i \in \{1,\dots,n\}$, the intensity of~$N_i$, is of the form $s  (t,x_i) =  u(t) v(x_i)$ where $u$ is an unknown density function on $\TT$, and $v$ some unknown function from $\XX$ into $[0,+\infty)$. This means, that in average, the number of failures of system~$i$,  $\E[N_i(\TT)] = v(x_i)$, depends on~$x_i$ through $v$ only, and conditionally to $N_i(\TT) = k_i > 0$, the times of failure are distributed along $\TT$ independently of $x_i$, but accordingly to the density $u$.

We shall therefore consider the class $\F$ defined by
\begin{eqnarray}\label{Fmultiplicativeclass}
\F = \left\{\kappa  v_1  v_2 , \, \kappa \geq 0 , \, (v_1,v_2) \in \L^2 (\TT,\mu) \times \L^2 (\XX,\nu_n) , \, \|v_1\|_{\t} = \|v_2\|_{\x}  = 1\right\} ,
\end{eqnarray}
which amounts to assuming that $s$ is of the form (or close to) a product function $u(t) v(x)$ with  $u = v_1^2$ and $v = \kappa^2 v_2^2$.

In this section, we introduce collections of models $\V_1$ and $\V_2$ in order to approximate  the components~$v_1$ and $v_2$ separately. Given $V_1 \in \V_1$ to approximate~$v_1$ and $V_2 \in \V_2$ to approximate~$v_2$, we approximate $v_1 v_2$ by the model  $V_1 \otimes V_2$ defined by
 \begin{eqnarray}\label{Vmultiplicatif}
 V_1 \otimes V_2 = \left\{ v_1' v_2' , \, (v_1',v_2') \in V_{1} \times V_{2} \right\}.
\end{eqnarray}
The metric dimension of $V_1 \otimes V_2$ is controlled as follows.
\begin{lemme} \label{lemmeproduitlinear}
Let $V_1$ and $V_2$ be a finite dimensional linear space of $\LL^2(\TT,\mu)$ and $\LL^2(\XX, \nu_n)$ respectively. The set $ V_1 \otimes V_2$ defined by~{(\ref{Vmultiplicatif})} has a finite metric dimension bounded by
$$D_{ V_1 \otimes V_2} = 1.4 \, (\dim V_{1} + \dim V_{2} + 1).$$
\end{lemme} 
By using  Theorem~\ref{thmselectionmodelegeneral}, we prove the following result. 

\begin{prop}  \label{propmultiplicativemodels}
Let $\V_{1}$ (respectively $\V_{2}$)  be an at most countable collection of finite dimensional linear spaces of $\L^2 (\TT, \mu)$  (respectively $\L^2 (\XX, \nu_n)$). Let, for all $i \in \{1,2\}$,  $\Delta_i$  be a non-negative mapping on~$\V_i$ such that   
 $$\sum_{ V_i \in \V_{_i}} e^{- \Delta_i (V_i)} \leq 1.$$
There exists an estimator $\hat{s}$ such that, for all $\kappa v_1 v_2 \in \F$, where $\F$ is defined by~{(\ref{Fmultiplicativeclass})}, 
\begin{eqnarray*} \label{inegaliteoraclemultiplicative}
C \E\left[ H^2(s,\hat{s}) \right] & \leq&   d_2^2 (\sqrt{s} , \kappa v_1 v_2 ) + \inf_{V_{1} \in \V_{1} }  \left\lbrace \kappa^2  d_{\t}^2 (v_1, V_{1}) + \frac{\dim V_{1} \vee 1 + \Delta_1 (V_{1})  }{n}\right\rbrace \nonumber \\
& & \quad + \inf_{V_{2} \in \V_{2} } \left\lbrace \kappa^2  d_{\x}^2 (v_2, V_{2}) + \frac{\dim V_{2} \vee 1 + \Delta_2 (V_{2})  }{n}\right\rbrace
\end{eqnarray*}
where $C$ is an universal positive contant. Furthermore, $\sqrt{\hat{s}}$ belongs to $\F$.
\end{prop}
Apart for the term $d_2^2 (\sqrt{s} , \kappa v_1 v_2 )$ which corresponds to some robustness with respect to the assumption $\sqrt{s} \in \F$, the risk bound we get corresponds to the one we would get if we could apply a model selection theorem on the components~$v_{1}$ and $v_{2}$ separately.

\subsection{Smoothness assumptions on $v_1$ and $v_2$.}
We illustrate this proposition by setting $\TT = [0,1]^{k_1}$, $\XX = [0,1]^{k_2}$, $\mu$ the Lebesgue measure and
\begin{eqnarray}  \label{FfonctionsproduitlinearHolder}
 \F =  \left\{\kappa v_1 v_2 ,  \kappa \geq 0 ,  v_1  \in \mathcal{H} ([0,1]^{k_1}) , \|v_1\|_{\t}  = 1 ,  v_2  \in \mathcal{H} ([0,1]^{k_2}) , \|v_2\|_{\x} = 1  \right\}.
\end{eqnarray}
We apply Proposition~\ref{propmultiplicativemodels} with families $\V_1$ and $\V_2$ of linear spaces possessing good approximation properties with respect to the functions of $\mathcal{H} ([0,1]^{k_1})$ and $\mathcal{H} ([0,1]^{k_2})$ respectively. This leads to  the following corollary.

\begin{cor}  \label{rateholdermultiplicative}
There exists an estimator $\hat{s}$ such that, for all $\kappa v_1 v_2 \in \F$, where $\F$ is defined  by (\ref{FfonctionsproduitlinearHolder}),  
\begin{eqnarray*}
C  \E\left[ H^2(s,\hat{s})  \right]  &\leq& d_2^2 ( \sqrt{s} , \kappa  v_1 v_2 )  +    \kappa^{\frac{2 k_1}{2 \bar{\boldsymbol{\alpha}}+k_1}}   L \left(v_1 \right)^{\frac{2 k_1}{2 \bar{\boldsymbol{\alpha}}+k_1}}   n^{- \frac{2 \bar{\boldsymbol{\alpha}}}{2 \bar{\boldsymbol{\alpha}} + k_1}}  \\
& & \qquad +   \kappa^{\frac{2 k_2}{2 \bar{\boldsymbol{\beta}}+k_2}}   L \left(v_2 \right)^{\frac{2 k_2}{2 \bar{\boldsymbol{\beta}}+k_2}}   n^{- \frac{2 \bar{\boldsymbol{\beta}}}{2 \bar{\boldsymbol{\beta}} + k_2}} + n^{-1}
\end{eqnarray*}
where $\boldsymbol{\alpha} \in (0,+\infty)^{k_1}$,  is such that $v_1  \in \mathcal{H}^{\boldsymbol{\alpha}} ([0,1]^{k_1})$, where $\boldsymbol{\beta} \in (0,+\infty)^{k_2}$   is such that $v_2  \in \mathcal{H}^{\boldsymbol{\beta}} ([0,1]^{k_2})$,
and where $C > 0$ depends only on $k_1$, $k_2$, $\max_{1 \leq j \leq k_1} \alpha_i$, and $\max_{1 \leq j \leq k_2} \beta_i$. 
\end{cor}
In particular, if $s$ is a product function of the form $\sqrt{s} = \kappa v_1 v_2 $ where $v_1 \in  \mathcal{H}^{\boldsymbol{\alpha}} ([0,1]^{k_1})$, and $v_2  \in \mathcal{H}^{\boldsymbol{\beta}} ([0,1]^{k_2})$, $\sqrt{s}$ is Hölderian with regularity $(\boldsymbol{\alpha}, \boldsymbol{\beta})$ on $[0,1]^{k_1+k_2}$.
However, the rate given by the corollary above is always faster than the one we would get by  Corollary~\ref{coorminimiaxrate2}  under smoothness assumption only.


\subsection{Mixing smoothness and structural assumptions.}
When $k_2$ is large, we  may consider structural assumptions on $v_2$ instead of smoothness ones to improve the risk bound. Proposition~\ref{propmultiplicativemodels} allows to consider a wide variety of situations thanks to the approximation results of~\cite{BaraudComposite} on composite functions. 
We do not present all of them for the sake of concisely. We just consider the example in which the class $\F$ is
\begin{eqnarray} \label{Fmultiplemodelindexmodel}
 \F  &=& \left\{  \kappa  v_1  v_2  , \,  \kappa \geq 0 , \, v_1 \in \mathcal{H} ([0,1]^{k_1}) , \, \ttheta_1,\dots,\ttheta_l \in \B(0,1)  ,  g \in \mathcal{H} ( \left[-1, 1 \right]^l )   , \right.  \\
& &  \left.  \forall \boldsymbol{x}  \in \XX ,\, v_2(\boldsymbol{x} ) = g \left( <\ttheta_1, \boldsymbol{x} >, \dots, <\ttheta_l,\boldsymbol{x}  >  \right)   , \, \|v_1\|_{\t} = \|v_2\|_{\x} = 1 \right\} \nonumber
\end{eqnarray}
where we have chosen $\TT = [0,1]^{k_1}$,  $\mu$  the Lebesgue measure and $$\XX = \B(0,1) = \left\{\boldsymbol{x} \in \R^{k_2}, \, \sum_{j=1}^{k_2} x_j^2 \leq 1 \right\}$$ the unit ball of $\R^{k_2}$.
The following corollary ensues from Proposition~\ref{propmultiplicativemodels} and Corollary~2 of~\cite{BaraudComposite}.
 
\begin{cor} \label{corsinglemodelratemultiplicatif}
There exists an estimator $\hat{s}$ such that, for all $\kappa  v_1 v_2 \in~\F$, where $\F$ is defined by~{(\ref{Fmultiplemodelindexmodel})},
\begin{eqnarray*}
C \E\left[ H^2\left(s,\hat{s} \right)  \right]  &\leq&   d_2^2 (\sqrt{s},   \kappa  v_1  v_2 )   +  \kappa^{\frac{2 k_1 } {2 \bar{\boldsymbol{\alpha}}+ k_1}} L(v_1)^{\frac{2 k_1 } {2 \bar{\boldsymbol{\alpha}}+ k_1}}  \, n^{- \frac{2 \bar{\boldsymbol{\alpha}}}{2 \bar{\boldsymbol{\alpha}} + k_1}} \\
& & \quad  + \kappa^{\frac{2  l } {2 \bar{\boldsymbol{\beta}}+ l}}  L(g)^{\frac{2 l  } {2 \bar{\boldsymbol{\beta}}+ l}}  \, n^{- \frac{2 \bar{\boldsymbol{\beta}} }{2 \bar{\boldsymbol{\beta}} + l}}  + \frac{ \ln \big( \kappa^2  \|g\|_{\boldsymbol{\beta}}^2 k_2^{-1} \big) \vee \ln n \vee 1}{n} k_2 
\end{eqnarray*}
where $\boldsymbol{\alpha} \in (0,+\infty)^{k_1}$, $\boldsymbol{\beta} \in (0,+\infty)^l$ are such that $v_1 \in \mathcal{H}^{\boldsymbol{\alpha}} ([0,1]^{k_1})$, $g \in \mathcal{H}^{\boldsymbol{\beta}}  ( \left[-1, 1 \right]^l )$ with  $v_2 (\boldsymbol{x} ) = g \left( <\ttheta_1, \boldsymbol{x} >,\dots, <\ttheta_l,\boldsymbol{x} >  \right)$ and where $C > 0$ depends only on $k_1$, $l$, $\boldsymbol{\alpha}$ and $\boldsymbol{\beta}$.
\end{cor}
When $\sqrt{s}$ belongs to the class $\F$, the risk bound  of the above inequality  corresponds to the one we would get if we could estimate the functions $v_1$ and $g$ separately. This risk bound is then better than the one we would get under smoothness assumptions on $v_2$ when $l < k_2$.

\subsection{Examples of parametric assumptions.} \label{sectionparametricproduit} 
Theorem~\ref{thmselectionmodelegeneral} also allows to deal with parametric assumptions. Hereafter, 
we consider a class~$\F$ of the form $$\F = \left\{ a u_{b} v_{\ttheta}, \, a \geq 0 , \, b \in I , \, \ttheta \in \R^{k_2} \right\}, $$
where $I$ is an interval of $\R$, $(u_b)_{b \in I}$ is a family of functions and $v_{\ttheta}$ is defined by $v_{\ttheta} (\boldsymbol{x} ) = \exp \left(<\boldsymbol{x}  , \ttheta> \right)$ for $\boldsymbol{x}  \in \XX = \{\boldsymbol{x}  \in \R^{k_2}, \sum_{j=1}^{k_2} x_j^2 \leq 1 \}$, the unit ball of~$\R^{k_2}$.
For each $i \in \{1,\dots n \}$, the intensity of $N_i$ is thus assumed to be proportional to an element of (or an element close to) some reference parametric model $\{u_b^2 , \, b \in I \}$. Let us give 3 examples of such models.

The Power Law Processes are Poisson processes whose intensities are proportional to $u_b(t) =  t^b$ for all  $t \in \TT = (0,1]$ and some  $b \in (-1/2,+\infty)$. Proposed first in~\cite{Duane1964}, this model is  popular in reliability. Indeed, although the intensity is simple, different situations can be modelled by this model.  For example, if $b =0$ each~$N_i$ obeys to an homogeneous Poisson process, whereas if $b > 0$ (respectively $b < 0$) the reliability of each system reduces (respectively improves) with  time.
In software reliability, we can cite the Goel-Okumoto model of~\cite{GoelOkumoto1979} and the $S$-Shaped model  of~\cite{Yamada1983}. The former considers intensities proportional to  $u_b(t) =  e^{- b t}$ whereas the latter corresponds to $u_b(t) =  \sqrt{t} e^{-b t}$ where $b \in [0,+\infty)$ and $t \in \TT = [0,+\infty)$. 


We consider the following assumption on the family  $\{u_b , \, b \in I \}$.
\begin{hyp}\label{assumptionparametricsansln}
The family $(u_{b})_{b \in  I}$ is a family of non vanishing functions of  $\L^2 (\TT, \mu)$ indexed by an interval $ I$ of the form $(b_0,+\infty)$. Moreover, there exists two positive non-increasing functions $ \underline{\rho}, \bar{\rho}$  on $I$, such that for all $b,b' \in  I$, 
$$ \underline{\rho} \left(b \vee b' \right) |b - b'| \leq \left\|\frac{u_{b}}{\|u_{b}\|_{\t}} - \frac{u_{b'}}{\|u_{b'}\|_{\t}} \right\|_{\t} \leq \bar{\rho} \left(b \wedge b' \right) |b - b'|.$$
\end{hyp}
The purpose of the lemmas below is to show that the above assumption  holds for the Duane, Goel-Okumoto and S-Shaped models.

\begin{lemme} \label{lemmepowerlawparametric1}
Let $ I = (-1/2,+\infty)$, $\TT = (0,1]$, $\mu$ the Lebesgue measure, and for $b \in  I$, $u_{b} (t) =  t^{b}$.
Assumption~\ref{assumptionparametricsansln} is satisfied with 
$$ \underline{\rho} (u) = \bar{\rho} (u) = \frac{1}{1 + 2 u} \quad \text{for all $u > -1/2$.}$$
\end{lemme}
\begin{lemme} \label{lemmepowerlawparametric2}
Let $ I = (0,+\infty)$, $\TT = [0,+\infty)$ , $\mu$ the Lebesgue measure, $k \in \N$,  and for $b \in  I$, $u_{b} (t) =  t^{k/2} e^{-b t}$.
Assumption~\ref{assumptionparametricsansln} is satisfied with 
$$\underline{\rho} (u) = \frac{1}{2 u} \quad \text{and} \quad \bar{\rho} (u) = \frac{\sqrt{k + 1}}{2 u} \quad \text{for all $u > 0$.}$$
\end{lemme}
All along this section, $\|\cdot\|$ denotes the standard Euclidean norm of $\R^{k_2}$ 
$$\forall \boldsymbol{x}   \in \R^{k_2}, \, \quad \|\boldsymbol{x}\|^2 = \sum_{j=1}^{k_2} x_j^2$$
and $d$ the distance induced by this norm.

\begin{prop}  \label{proppowerlawundercox}
Let $(u_{b})_{b \in  I}$ be a family such that Assumption~\ref{assumptionparametricsansln} holds. 
There exist $\hat{a} \geq 0$, $\hat{b} \in I$ and $\tttheta \in \R^{k_2}$,  such that the estimator~$\hat{s} = (\hat{a} u_{\hat{b}} v_{\tttheta})^2$ satisfies, for all $a \geq 0$, $b \in  I$, $\ttheta \in \R^{k_2}$, and $f \in \F$ of the form $f(t,\boldsymbol{x} ) = a u_b(t) v_{\ttheta} (\boldsymbol{x} )$, 
\begin{eqnarray} \label{relapowerlawundercox}
C \E\left[ H^2(s,\hat{s}) \right] \leq    d_2^2 \left(\sqrt{s},  f \right)  + \frac{k_2 \left(1 \vee \|\ttheta\| \right)}{n} + \frac{C'}{n}
\end{eqnarray}
where $C$ is an universal positive constant and where $C'$ depends only on $\underline{\rho}$, $\bar{\rho}$, $b_0$ and  $b$. More precisely,
$$C' =   \log \left[1 \vee \bar{\rho} \left(b_0 + \frac{b-b_0}{b-b_0+1} \right) \right] + \left|\log \left(1 \wedge \underline{\rho} (1 + b) \right) \right| + \left|\log (b - b_0) \right|. $$
\end{prop}
Under parametric assumptions on $s$, this result says that the rate of convergence of~$\hat{s}$ is of order $n^{-1}$, which is quite satisfying when $n$ is large, but may be inadequate  in a non-asymptotic point of view.
Indeed, the second term of the right-hand side of inequality~{(\ref{relapowerlawundercox})} may be large  especially when $k_2$ is large, says larger than $n$. This difficulty can be overcome by considering that $\ttheta$ is sparse, which means that $\ttheta$ is close to some (unknown) linear subspace $W$ of $\R^{k_2}$ with $\dim W$ small. Below, we generalize Proposition~\ref{proppowerlawundercox} to take account of this situation.

\begin{prop}  \label{proppowerlawundercox2}
Let $(u_{b})_{b \in  I}$ be a family such that Assumption~\ref{assumptionparametricsansln} holds. 
Let $\W$  be an at most countable family of linear subspaces of $\R^{k_2}$ and let $\Delta$ be a non-negative map on~$\W$ such that $\sum_{W \in \W} e^{-\Delta(W)} \leq 1$.

There exist $\hat{a} \geq 0$, $\hat{b} \in I$ and $\tttheta \in \R^{k_2}$,  such that the estimator~$\hat{s} = (\hat{a} u_{\hat{b}} v_{\tttheta})^2$ satisfies, for all $a \geq 0$, $b \in  I$, $\ttheta \in \R^{k_2}$, and $f \in \F$ of the form $f(t,\boldsymbol{x} ) = a u_b(t) v_{\ttheta} (\boldsymbol{x} )$, 
\begin{eqnarray*}
C \E\left[ H^2(s,\hat{s}) \right] &\leq&   d_2^2 (\sqrt{s},  f ) +  \frac{C'}{n}   \\
& &  \!\!\!\! \!\!\!\! \!\!\!\!\!\! +   \inf_{W \in \W } \left\lbrace   a^2 \|u_b\|_{\t}^2 e^{2 \|\theta\|} d^2 (\ttheta,W)  +  \frac{(1 \vee \dim W) (1 \vee \|\ttheta\|)   + \Delta(W) }{n}    \right\rbrace 
\end{eqnarray*}
where $C$ is an universal positive constant and where $C'$ is given by Proposition~\ref{proppowerlawundercox}.
\end{prop}

For illustration purpose, let us make explicit the constant $C'$ for the Duane model, and let us therefore assume that there exist some unknown parameters $a, b, \ttheta$ such that $s$ is of the form  $\sqrt{s (t,\boldsymbol{x} )} =  a t^b  \exp \left(<\ttheta, \boldsymbol{x} > \right)$.
We derive from Proposition~\ref{proppowerlawundercox} an estimator whose risk satisfies
\begin{eqnarray}\label{riskexplicitpowerlaw}
C \E\left[ H^2(s,\hat{s}) \right] \leq \frac{\left(1 \vee \|\ttheta\| \right) k_2 + |\log(2 b + 1)|}{n}
\end{eqnarray}
where $C$ is an universal positive constant. 
 However, if for instance $k_2$ is large and if most of the components of $\ttheta$ are small or null, the preceding proposition can be used to improve substantially the risk of our estimators.  For simplicity, assume that
 $$k_{\star} = \left| \left\{ j \in \{1,\dots , k_2 \}, \theta_j \neq 0 \right\} \right| $$
 is small. We then define    the set $\M$ of all subsets of $\{1,\dots, k_2\}$, and for each $m \in \M$, the set
 $$W_m = \left\{ (y_1,\dots, y_{k_2}) , \, \forall j \not \in m , \, y_j = 0 \right\} \subset \R^{k_2}.$$
 We apply Proposition~\ref{proppowerlawundercox2} with 
  $$\W  = \left\{W_m , \, m \in \mathcal{M} \right\} \quad \text{and} \quad \forall m \in \M, \;\Delta(W_m) = 1 + |m| + \log \binom {k_2}{|m|}. $$
This leads to an estimator $\hat{s}$ such that
$$C'' \E\left[ H^2(s,\hat{s}) \right] \leq    \frac{(1 \vee \log k_2 \vee \|\ttheta\|) (1 \vee k_{\star})  + |\log(2 b + 1)|}{n}     ,$$ 
which improves inequality~{(\ref{riskexplicitpowerlaw})} when $k_{\star}$ is small and  $k_2$ large. 

\section{Parametric models} \label{sectionrobuste}
In this section, we consider the natural situation where the intensity of each process~$N_i$ belongs (or is close) to a same parametric model. Throughout this section, $n \geq 2$.
 Let us consider a closed rectangle $\Theta$ of $\R^k$, that is a subset of $\R^k$ for which there exist $m_1,\dots,m_k \in \R \cup \{-\infty\}$ and $M_1,\dots,M_k \in \R \cup \{\infty\}$ such that
 $$\Theta = \left\{ \boldsymbol{x} \in \R^k, \, \forall i \in \{1,\dots,k\}, \, m_i \leq x_i \leq M_i \right\}.$$
 Let us denote by  $\mathcal{F} = \left\{ f_{\ttheta} , \, \ttheta \in \Theta \right\}$ a class of functions of $\LL^2(\TT, \mu)$.  Our aim is to estimate~$s$ when, for each $i \in \{1,\dots n\}$, the square root of the intensity of the Poisson process~$N_i$, $\sqrt{s(\cdot, x_i)}$, is (or is close to) an element of $\mathcal{F}$. We introduce thus the class of functions~$\F$ defined by
\begin{eqnarray*}\label{Frobuste}
\F = \left\{(t,x) \mapsto f_{\boldsymbol{u} (x)} (t) , \, \text{where $\boldsymbol{u}$ is a map from $\XX$  into $\Theta$} \right\}.
\end{eqnarray*}
For instance, if $\mathcal{F}$  corresponds to the Duane model (see Section~\ref{sectionparametricproduit}), $\Theta$ is a closed rectangle included in $\R \times (-1/2,+\infty)$ and 
$$\mathcal{F} = \left\{a t^b, \, (a,b) \in \Theta \right\}.$$
The class $\F$ is  then the set of all functions $f$ of the form $f(t,x) = a(x) t^{b(x)}$ where~$a$ and~$b$ are two functions on $\XX$ such that $(a(x),b(x)) \in \Theta$ for all $x \in \XX$.

We consider the following assumption to deal with more general classes $\mathcal{F}$.
\begin{hyp} \label{assumptionFRobuste}
The set $\Theta$ is a closed rectangle of $\R^k$. There exist $\boldsymbol{\alpha}  = (\alpha_j)_{1 \leq j \leq k} \in (0,1]^k$ and $\boldsymbol{R} = (R_j)_{1 \leq j \leq k}  \in (0,+\infty)^k$  such that 
\begin{eqnarray} \label{equationFRobuste}
 \forall \ttheta, \ttheta' \in \Theta , \quad \|f_{\ttheta} - f_{\ttheta'}\|_{\t} \leq \sum_{j=1}^k R_{j}  \, |\theta_j - \theta_j'|^{\alpha_j}.
\end{eqnarray}
\end{hyp}
The aim of the lemmas below is to prove that this assumption is satisfied for the Duane, Goel-Okumoto and S-Shaped models.

\begin{lemme} \label{LemmeDuaneModelParametric}
Let $\mu$ be the Lebesgue measure, and for all $\ttheta \in \R \times [-1/2,+\infty)$, 
\begin{eqnarray*} 
f_{\ttheta} (t) = \theta_1 t^{\theta_2} \quad \text{for all $t \in \TT = (0,1]$.}
\end{eqnarray*}
Then, for all positive numbers $r_1,r_2$, and all $\ttheta, \ttheta' \in [-r_1,r_1] \times \left[- 1/2+ 1/r_2 , + \infty \right)$,
\begin{eqnarray*}
 \|f_{\ttheta} - f_{\ttheta'}\|_{\t} \leq r_2^{1/2} |\theta_1 - \theta'_1| + \sqrt{2} r_1 r_2^{3/2} |\theta_2 - \theta'_2|.
\end{eqnarray*}
\end{lemme}
\begin{lemme} \label{LemmeDuaneModelParametric2}
Let $\mu$ be the Lebesgue measure, and for all $k \in \{0,1\}$,  $\ttheta = ( \theta_1,  \theta_2) \in \R \times (0,+\infty)$,
\begin{eqnarray*} 
f_{\ttheta} (t) =  \theta_1 t^{k/2} e^{- \theta_2 t } \quad \text{for all $t \in \TT = (0,+\infty)$.} 
\end{eqnarray*}
Let $r_1, r_2$ be two positive numbers and let us set
\begin{eqnarray*}
C_1 (0) =  \left(r_2/2\right)^{1/2}  \quad C_2 (0) = {r_1 r_2^{3/2}}/2 \quad C_1 (1) =  r_2/2 \; \text{and} \quad  C_2 (1) =\left(3/8\right)^{1/2} r_1 r_2^2.
\end{eqnarray*} 
For  all $\ttheta, \ttheta' \in [-r_1,r_1] \times \left[1/r_2 , + \infty \right)$,
\begin{eqnarray*}
\|f_{\ttheta} - f_{\ttheta'}\|_{\t} \leq C_1 (k) |\theta_1 - \theta'_1| + C_2(k) |\theta_2 - \theta'_2|.
\end{eqnarray*}
\end{lemme}

\paragraph{Remark.} For the Duane model, Lemma~\ref{LemmeDuaneModelParametric} shows that Assumption~\ref{assumptionFRobuste} is  fulfilled for $\Theta = [-r_1,r_1] \times \left[- 1/2+ 1/r_2 , + \infty \right)$. This will allow to obtain a  risk bound when $\sqrt{s}$ is close to the class
 \begin{eqnarray*}
\F_{r_1,r_2} &=& \left\{(t,x) \mapsto a(x) t^{b(x)}, \text{where $a$ maps $\XX$ into $[-r_1,r_1]$ and } \right. \\
& & \left. \quad \quad \quad \quad \quad \quad \quad \text{$b$ maps $\XX$ into $[-1/2 + 1/r_2, +\infty)$} \right\}.
\end{eqnarray*}
By using Proposition~\ref{propmixing} with $\mathfrak{F} = \{ \F_{r_1,r_2},\, r_1,r_2 \in \N^{\star} \}$, we can also derive a risk bound for the class 
\begin{eqnarray*}
\F &=& \bigcup_{r_1,r_2 \in \N^{\star}} \F_{r_1,r_2}\\
&=& \left\{(t,x) \mapsto a(x) t^{b(x)}, \text{where $a$ maps $\XX$ into a compact subset of $\R$, } \right. \\
& & \qquad \left. \text{ and $b$ maps $\XX$ into a closed interval  included in $(-1/2, +\infty)$} \right\}.
\end{eqnarray*}

\subsection{A model selection theorem.} The main theorem of Section~\ref{sectionrobuste} is the following.
\begin{thm} \label{thmestimationrobuste}
Suppose that Assumption~\ref{assumptionFRobuste}  holds. Let $\W_1,\dots,\W_k$ be $k$ families of finite dimensional linear subspaces of $\L^2 (\XX,\nu_n)$. Let for each $j \in \{1,\dots,k\}$, $\Delta_j$ be a non-negative mapping on $\W_j$ such that $\sum_{W_j \in \W_j} e^{- \Delta_j (W_j)} \leq 1$.

There exists an estimator~$\hat{s}$ such that for all map $\boldsymbol{u} = (u_1,\dots,u_k)$ from $\XX$ with values into $\Theta$, and $f \in \F$  of the form $f(t,x) = f_{\boldsymbol{u} (x)} (t)$,
\begin{eqnarray*} \label{eqestmiationrobuste}
C \E\left[ H^2(s,\hat{s})  \right]  \leq    d_2^2 ( \sqrt{s} , f ) +  \sum_{j=1}^k \varepsilon_j (u_j) 
\end{eqnarray*}
where $\varepsilon_j (u_j)$ is defined by
\begin{eqnarray*}
\varepsilon_j (u_j) = \inf_{W_j \in \W_j} \left\{ R_{j}^2 \, \left(d_{\x} \left(u_j, W_j \right) \right)^{2 \alpha_j} +   \frac{\left(\dim (W_j) \vee 1 \right) \tau_{\boldsymbol{u}, j} (n) + \Delta_j (W_j)}{n} \right\}, 
\end{eqnarray*}
where $$\tau_{\boldsymbol{u}, j} (n) = \log  n  + \log (1 \vee  R_{j}) + \log \left( 1 \vee \|u_j\|_{\x} \right), $$ 
and where $C > 0$ depends only on $k$ and $\alpha_1,\dots , \alpha_k$.
\end{thm}
Roughly speaking, this result says that the risk bound we get  when $\sqrt{s}$ is of the form $\sqrt{s (t, x)} = f_{\boldsymbol{u} (x)} (t)$, corresponds to the one we would get if we could apply a model selection theorem on the components $u_1,\dots , u_k$  separately.
Each term $\varepsilon_j (u_j) $ can be controlled under  structural or smoothness assumptions on $u_j$.
 For instance, if $\XX = [0,1]^{k_2}$ and if $u_j$ is assumed to belong to the class $\F_j =   \mathcal{H} ([0,1]^{k_2}) $, a suitable choice of $(\W_j$, $\Delta_j)$ leads to 
\begin{eqnarray*}
C_j \varepsilon_j (u_j) \leq \left(R_{j} L(u_j)^{\alpha_j}\right)^{\frac{2 k_2}{k_2 + 2 \alpha_j \bar{\boldsymbol{\beta}}_j}} \left(\frac{\tau_{\boldsymbol{u}, j} (n) }{ n} \right)^{\frac{2 \alpha_j \bar{\boldsymbol{\beta}}_j}{2\alpha_j \bar{\boldsymbol{\beta}}_j + k_2}}  +   \frac{\tau_{\boldsymbol{u}, j} (n) }{ n} 
\end{eqnarray*}
where $\boldsymbol{\beta}_j$ is such that $u_j \in \mathcal{H}^{\boldsymbol{\beta}_j }([0,1]^{k_2})$ and where $C_j > 0$ depends only on~$k_2$ and~$\boldsymbol{\beta}_j$.
In particular, if $\alpha_j  = 1$ and if $n$ is large, $\varepsilon_j (u_j)$ is of order $ (\log n / n)^{2 \bar{\boldsymbol{\beta}}_j / (2 \bar{\boldsymbol{\beta}}_j + k_2)}$. Apart from the logarithmic factor, this corresponds to the estimation rate of an Hölderian function on $[0,1]^{k_2}$.

The corollary below illustrates this result for the Duane model.
\begin{cor} \label{corrobustepowerlaw}
There exists an estimator $\hat{s}$ such that, for all $\boldsymbol{\alpha} \in (0,+\infty)^{k_2}$, $\boldsymbol{\beta} \in (0,+\infty)^{k_2}$,  for all  $a \in \mathcal{H}^{\boldsymbol{\alpha}} ([0,1]^{k_2})$, $b \in \mathcal{H}^ {\boldsymbol{\beta}} ([0,1]^{k_2})$ satisfying $b > -1/2$, and for all function~$f$ of the form $f(t,x) = a(x) t^{b(x)}$,
\begin{eqnarray*}
C \E\left[ H^2(s,\hat{s}) \right] &\leq&  d_2^2 \left(\sqrt{s} ,   f \right)  \\
& &   +  \left( \frac{ 1 }{1 \wedge \inf_{x  \in [0,1]^{k_2}} (2 b(x ) + 1 ) }\right)^{\frac{k_2}{2 \bar{\boldsymbol{\alpha}} + k_2}}  L(a)^{\frac{2 k_2 } {2 \bar{\boldsymbol{\alpha}}+k_2}}  \left(\frac{\log  n}{n}\right)^{\frac{2 \bar{\boldsymbol{\alpha}}}{2 \bar{\boldsymbol{\alpha}} + k_2}}   \\
& &   +\left( \frac{ 1 \vee \|a\|_{\infty}^2}{1 \wedge \inf_{x  \in [0,1]^{k_2}} (2 b(x ) + 1 )^3 }\right)^{\frac{k_2}{2 \bar{\boldsymbol{\beta}} + k_2}} L(b)^{\frac{2 k_2 } {2 \bar{\boldsymbol{\beta}}+k_2}}  \, \left(\frac{\log  n}{n}\right)^{\frac{2 \bar{\boldsymbol{\beta}}}{2 \bar{\boldsymbol{\beta}} + k_2}} \\
& &   + C'  \frac{\log n }{n} 
\end{eqnarray*}
where $C > 0$ depends on $k_2$, $\max_{1 \leq j \leq k} \alpha_j$, $\max_{1 \leq j \leq k} \beta_j$, and where  $C'$ depends on $L(a)$, $L(b)$, $\bar{\boldsymbol{\alpha}}$, $\bar{\boldsymbol{\beta}}$, $\|a\|_{\infty}$, $\|b\|_{\infty}$ and $\inf_{x \in [0,1]^{k_2}} (2 b(x) + 1)$.
\end{cor}

\subsection{Change point detection.}
In the case where the intensity $s_i$ of each $N_i$ is of the form $\sqrt{s_i (t)} = f_{\ttheta_i} (t)$, a natural way to control the risk of our estimator $\hat{s}$ is to consider some assumptions on the map $i \mapsto \ttheta_i$. This problem amounts to choosing suitable collections  $\W_1,\dots, \W_k$ to approximate functions on $\XX = \{1,\dots, n\}$. 

In this section, we focus on the case where the map   $i \mapsto \ttheta_i$ is piecewise constant with a small number of jumps.
Let $\mathcal{P}$ be the set of partitions of $\{1,\dots,n\}$ into intervals. We aim at estimating $s$ when 
there exists a partition $P_0 \in \mathcal{P}$ such that $s$ is of the form 
\begin{eqnarray} \label{eqformschangepoint}
\forall I \in P_0, \exists \ttheta_I \in \Theta, \forall i \in I, \quad  \sqrt{s_i (t)} = f_{\ttheta_I}  (t)  \quad \text{for all $t \in \TT$.}
\end{eqnarray}
We define  for each partition $P \in \mathcal{P}$, the linear space  of piecewise constant functions
$$W_P = \left\{\sum_{I \in P} a_I \1_I , \, a_I \in \R \right\} $$
and apply Theorem~\ref{thmestimationrobuste} with the collections  and  maps  defined by 
$$\forall j \in \{1,\dots, k\}, \; \W_j = \left\{ W_P , \, P \in \mathcal{P} \right\} \; \text{and} \; \Delta_j (W_P) = |P| + \log \binom {n-1}{|P|-1}.$$
This leads to the result below.
\begin{cor}  \label{corchangepoint}
Assume that Assumption~\ref{assumptionFRobuste}  and relation~{(\ref{eqformschangepoint})} hold.  
There exists an estimator~$\hat{s}$ such that
\begin{eqnarray*}
C \E\left[ H^2(s,\hat{s})  \right]  \leq    |P_0|   \frac{\log n +  C'}{n},
\end{eqnarray*}
where $C > 0$ depends only on $k$ and $\alpha_1,\dots , \alpha_k$, where $C'$ is given by
$$C' = \sup_{1 \leq j  \leq k} \left( \log \left( 1 + R_j \right) \right) + \sup_{I \in P} \left( \log \left(1 +  \| \ttheta_I \|_{\infty}  \right) \right),$$
and where $\| \ttheta_I \|_{\infty} = \sup_{1 \leq j \leq k} |(\ttheta_I)_j|$.
\end{cor}
For illustration purpose, in the context of the Duane model, there exist  $a_1\dots, a_n \in (0,+\infty)$, and  $b_1,\dots, b_n \in (-1/2, + \infty)$ such that $\sqrt{s_i(t)} = a_i t^{b_i}$ for all $t \in (0,1]$. By combining the preceding corollary with Proposition~\ref{propmixing}, we build an estimator $\hat{s}$ such that
\begin{eqnarray*}
C \E\left[ H^2(s,\hat{s})  \right]  \leq  \left(1 + r_1 + r_2 \right) \frac{\log n + C'}{n}  
\end{eqnarray*}
where $r_1$ and $r_2$ are the numbers  of jumps of the maps $i \mapsto a_i$ and $i \mapsto b_i$ respectively, where~$C$ is an universal positive constant, and where $C'$ depends on $\sup_{1 \leq i \leq n} a_i$, $\sup_{1 \leq i \leq n} |b_i|$ and $\inf_{1 \leq i \leq n} (2 b_i + 1)$.

The preceding collections $\W_1,\dots, \W_k$ can also be used to approximate the map $i \mapsto \ttheta_i$ under other assumptions such as smoothness ones. For instance, an approximation theorem for monotone functions on $\{1,\dots,n\}$ can be found in~\cite{BaraudBirgeHistogramme} and can be used to deal with the situation where some components of the map $i \mapsto \ttheta_i$ are monotone.

\section{Proofs} \label{sectionpeuves}

\subsection{Proof of Theorem~\ref{thmselectionmodelegeneral}.}
Throughout the proof, we set $\mathbf{N} = (N_1,\dots,N_n)$ and $\mathbf{x} = (x_1,\dots,x_n)$.

\subsubsection{About the $T$-estimators.} \label{SectionAboutTEstimators}
We begin to briefly recall the general strategy  introduced in~\cite{BirgeTEstimateurs} to build estimators from tests. 

Given two  distinct functions $f, f'$ of  $\L_{+}^1 (\TT \times \XX, M)$, a test function  $\psi_{f,f'} (\mathbf{N}, \mathbf{x})$ is  a measurable function with values in $\{f,f'\}$. The convention is that $\psi_{f,f'} (\mathbf{N}, \mathbf{x}) = f$ means accepting $f$ whereas $\psi_{f,f'} (\mathbf{N}, \mathbf{x}) = f'$ means accepting $f'$.
In what follows, we need tests with the following properties. We shall build them in Section~\ref{SectionConstructionTest}.

\begin{hyp} \label{hypotheseExistenceTests}
There exist $a > 0$, $\kappa > 0$ such that for  all distinct functions $f, f' \in   \L_{+}^1 (\TT \times \XX, M)$ and all $z \in \R$, there exists a test $\psi_{f,f'}^{(z)} (\mathbf{N}, \mathbf{x})$ satisfying
\begin{eqnarray}
\sup_{\substack{f \in  \L_{+}^1 (\TT \times \XX, M), \\ \kappa H (s, f) \leq H (f,f')}} \P \left[  \psi_{f,f'}^{(z)} (\mathbf{N}, \mathbf{x}) = f'\right] &\leq&  \exp \left[- a n \left(H^2 (f,f') + z \right) \right] \label{Eq1Test} \\
\sup_{\substack{f \in  \L_{+}^1 (\TT \times \XX, M), \\ \kappa H (s, f') \leq H (f,f')}} \P \left[  \psi_{f,f'}^{(z)} (\mathbf{N}, \mathbf{x}) = f\right] &\leq& \exp \left[- a n \left(H^2 (f,f') - z \right) \right] \label{Eq2Test}.
\end{eqnarray}
\end{hyp}
We now consider an at most countable collection $\S$ of subsets of  $\LL^1_+ (\TT \times \XX, M)$. We shall assume that the sets $S \in \S$ are $D$-models. We recall the definition below.
\begin{definition}
 A subset $S$ of $\LL^1_+ (\TT \times \XX, M)$ is called a $D$-model with parameters $\bar{\eta}_S$, $\bar{D}_S$ and~$1$ if
$$\left|S \cap \B(f, x \bar{\eta}_S) \right| \leq \exp \left[\bar{D}_S x^2 \right] \quad \text{for all $x \geq 2$ and $f \in \LL^1_+ (\TT \times \XX, M)$,}$$
where $\B (f, x \bar{\eta}_S)$ is the closed ball centered at $f$ with radius $x \bar{\eta}_S$ of the metric space $(\LL^1_+ (\TT \times \XX, M), H)$.
\end{definition}
The tests allow  to select among the functions of $\cup_{S \in \S} S$. Precisely, the  selection rule  is the following.

Given a collection $\S$ of $D$-models, we set for all $f \in \cup_{S \in \S} S$, 
$$\bar{\eta} (f) = \inf \left\{ \bar{\eta}_S , \, S \in \S, \, S \ni f \right\}$$
and for all $f' \in \cup_{S \in \S} S$,  $f' \neq f$, $z_{f,f'} = \bar{\eta} (f')^2 - \bar{\eta}(f)^2$.
We define for all $f \in \cup_{S \in \S} S$, 
$$\mathcal{R} (f) = \left\{f' \in \cup_{S \in \S} S , \; \psi_{f,f'}^{(z_{f,f'})} (\mathbf{N}, \mathbf{x}) = f'  \right\} $$
and consider
$$
\gamma(f) =  \left \{
\begin{array}{l l}
     \sup \left\{ H^2 (f,f'), \; f' \in \mathcal{R} (f) \right\} & \quad \text{if $\mathcal{R} (f) \neq \emptyset$,} \\
    0 & \quad \text{if $\mathcal{R} (f) = \emptyset$.} \\
\end{array}
\right.$$
Given $\varepsilon > 0$,  a $T_{\varepsilon}$-estimator is a measurable function $\hat{s} = \hat{s}  (\mathbf{N}, \mathbf{x})$ with values in  $\cup_{S \in \S} S$ such that
$$\gamma  (\hat{s}) \vee \varepsilon \bar{\eta} (\hat{s}) = \inf_{ f \in\cup_{S \in \S} S} \left[ \gamma(f) \vee \varepsilon \bar{\eta}(f) \right].$$

Theorem~5 of~\cite{BirgeTEstimateurs} shows that such a minimizer exists almost surely and they all possess similar theoretical properties. In our framework, we can rewrite it as follows.

\begin{thm} \label{TheoremTEstimateur}
Suppose that Assumption~\ref{hypotheseExistenceTests} holds.
Let  $\S$ be an at most countable collection of $D$-models such that $\bar{D}_{S} \geq 1/2$ for all $S \in \S$, 
$$\sum_{S \in \S} \exp \left(- \frac{a n  \bar{\eta}_S^2 }{21} \right) \leq 1 \quad \text{and} \quad a n \bar{\eta}_S^2 \geq \frac{21 \bar{D}_S}{5} \quad \text{for all $S \in \S$.}$$
For all $\varepsilon \in (0,4]$, there exists almost surely a $T_{\varepsilon}$-estimator~$\hat{s} \in \L^1_+ (\TT \times \XX, M)$. Moreover, any of them satisfies  
\begin{eqnarray*}\label{eqnoraclegeneralnonintegree}
\P \left[ C H^2(s, \hat{s})  \geq   \inf_{S \in \S} \left\lbrace H^2 \left(s,S \right)   + \bar{\eta}_S^2 \right\rbrace + \xi \right] \leq e^{-n  \xi} \quad \text{for all $\xi > 0$,}   
\end{eqnarray*}
where $C > 0$ depends only on $a, \kappa$.
\end{thm}
It remains thus to construct the tests and the collection $\S$ to prove Theorem~\ref{thmselectionmodelegeneral}.

\subsubsection{Definition of the tests} \label{SectionConstructionTest}
Our tests are inspired from the variational formula in~\cite{BaraudMesure}. Let, for all functions $f, f'$ of $\LL^1_+ (\TT \times \XX, M)$, $T_{f,f'} (\mathbf{N},\mathbf{x})$ be the functional
\begin{eqnarray*}\label{test}
T_{f,f'} (\mathbf{N},\mathbf{x}) &=&  \frac{1}{2 n}  \sum_{i=1}^n   \int_{\TT} \sqrt{\frac{f (t,x_i) +f' (t,x_i)}{2}} \left( \sqrt{f' (t,x_i)} - \sqrt{f (t,x_i)}  \right) \d \mu (t)  \\
& & \quad +  \frac{1}{\sqrt{2} n}  \sum_{i=1}^n    \int_{\TT} \frac{\sqrt{f' (t, x_i)} - \sqrt{f (t, x_i)} }{\sqrt{f (t,x_i) + f' (t,x_i)}} \d N_i(t)   \\ 
& & \quad - \frac{1}{2 n}  \sum_{i=1}^n    \int_{\TT}   \left( f' (t, x_i) - f (t, x_i)  \right)  \d \mu(t)
\end{eqnarray*}
where the convention $0/0$ is in use.
We prove the following.

\begin{lemme} \label{lemmatestrobuste}
There exist positive numbers $a,b$ such that for all $z \in \R$, and all $f, f' \in \LL^1_+ (\TT \times \XX, M)$ satisfying $4 H(s, f) \leq H(f,f')$,
$$\P \left[ T_{f,f'} (\mathbf{N},\mathbf{x})\geq b z \right] \leq \exp \left[ -n  a  \left( H^2(f,f') + z \right) \right].$$
\end{lemme}
The proof of this lemma is delayed to Section~\ref{SectionProoflemmaDModels} and we refer to the proof for the exact values of~$a$ and~$b$.

This lemma says that the functional $T_{f,f'} (\mathbf{N},\mathbf{x})$ can be used to construct  the tests. Precisely, we  set for all $z \in \R$, $f, f' \in \LL^1_+ (\TT \times \XX, M)$, $f \neq f' $,
\begin{equation*}
 \psi_{f,f'}^{(z)}  (\mathbf{N}, \mathbf{x}) =
  \begin{cases}
  f' & \text{if $T_{f,f'} (\mathbf{N},\mathbf{x})> b z$ } \\
   f & \text{if $T_{f,f'} (\mathbf{N},\mathbf{x})< b z$,} 
  \end{cases}
\end{equation*}
and $\psi_{f,f'}^{(z)}  (\mathbf{N}, \mathbf{x})$ is defined arbitrary  in case of equality. Thanks to the above lemma, (\ref{Eq1Test}) holds. Note that  (\ref{Eq2Test}) also holds since $T_{f,f'} (\mathbf{N},\mathbf{x}) = -T_{f',f} (\mathbf{N},\mathbf{x})$.

\subsubsection{Construction of $\S$}
The collection $\S$ is derived from~$\V$. We shall show in Section~\ref{SectionProoflemmaDModels} the following lemma.
\begin{lemme} \label{lemmaDModels}
For all $\eta > 0$ and $V \in \V$ there exists a $D$-model  $\bar{S}_V (\eta)$ with parameters $\eta$, $63 D_V (\eta/2)$ and $1$. Moreover,
\begin{eqnarray}\label{eqlemmapreuvetheoremprincipal}
H (s, \bar{S}_{V} (\eta)) \leq 2 \sqrt{2} \left( d_2 \left(\sqrt{s}, V \right) + \eta \right)
\end{eqnarray}
and for all $f \in \bar{S}_V (\eta)$, there exists $g \in V$ such that $\sqrt{f} = g \vee 0$.
\end{lemme}
Please note that we can assume (for the sake of simplicity and with no loss of generality), that $D_V$ is non-increasing.
We then set for all $V \in \V$,
 $$\bar{\eta}_{\bar{S}_V} =  \left(21 \sqrt{ \frac{3 }{5 a }    } \eta_V \right) \vee \sqrt{\frac{21 \Delta(V) }{ n   a   } } \quad \text{and} \quad \bar{S}_V = \bar{S}_V (\bar{\eta}_{\bar{S}_V})$$
 where $a$ is given by Lemma~\ref{lemmatestrobuste}. Actually $a$ is very small (smaller than $1$), which implies that $\bar{\eta}_{\bar{S}_V}/2 \geq \eta_V$ and thus
 $$63 D_V (\bar{\eta}_{\bar{S}_V}/2) \leq 63 D_V (\eta_V).$$
 Consequently, the set $\bar{S}_V$ is a $D$-model with parameters $\bar{\eta}_{\bar{S}_V}$, $\bar{D}_{\bar{S}_V} =  63 D_V (\eta_V)$ and $1$.
The collection~$\S$ is then defined by $\S = \left\{\bar{S}_V, \, V \in \V \right\}$.

\subsubsection{Proof of Theorem~\ref{thmselectionmodelegeneral}.}
The assumptions of Theorem~\ref{TheoremTEstimateur} are fulfilled: 
$$a n \bar{\eta}_{\bar{S}_V}^2 \geq  \frac{21^2 \times 3}{5} n \eta_V^2 \geq    \frac{21^2 \times 3}{5}  D_V (\eta_V) \geq \frac{21 \bar{D}_{\bar{S}_V}}{5}  $$ 
and
\begin{eqnarray*}
\sum_{V \in \V} \exp \left(- \frac{a  n  \bar{\eta}_{S_V}^2 }{21} \right)  \leq  \sum_{V \in \V} \exp \left(- \Delta(V) \right) \leq 1.
\end{eqnarray*}
The selection rule described in Section~\ref{SectionAboutTEstimators} provides thus an estimator $\hat{s}  \in \cup_{V \in \V} \bar{S}_V$ such that, for all $\xi > 0$, 
\begin{eqnarray*}
\P \left[ C H^2(s, \hat{s})    \geq     \inf_{V \in \V} \left\lbrace H^2 \left(s, \bar{S}_V \right)   + \bar{\eta}_{\bar{S}_V}^2  \right\rbrace + \xi \right] \leq e^{-n \xi}
\end{eqnarray*}
where $C > 0$ is universal.
By using  inequality~{(\ref{eqlemmapreuvetheoremprincipal})},  
$$ H^2 \left(s, \bar{S}_V \right) \leq 16 \left[ d_2^2(\sqrt{s},V) + \bar{\eta}_{\bar{S}_V}^2 \right]$$
and hence
$$\inf_{V \in \V} \left\lbrace H^2 \left(s, \bar{S}_V \right)   + \bar{\eta}_{\bar{S}_V}^2  \right\rbrace  \leq C' \inf_{V \in \V} \left\lbrace d^2_2 \left(\sqrt{s}, V \right)   + \eta_{V}^2 + \frac{\Delta(V)}{n} \right\rbrace  $$
for some universal constant $C' > 0$.
Finally,
\begin{eqnarray*}
\P \left[ C'' H^2(s, \hat{s})    \geq      \inf_{V \in \V} \left\lbrace d^2_2 \left(\sqrt{s}, V \right)   + \eta_{V}^2 + \frac{\Delta(V)}{n} \right\rbrace + \xi \right] \leq e^{-n \xi}
\end{eqnarray*}
where $C'' = C/(C' \vee 1)$.\carre

\subsubsection{Proof of Lemma~\ref{lemmatestrobuste}.} \label{sectionpreuvelemmerobuste}
We start with the following Bennett-type inequality which generalizes Proposition~7 of~\cite{Reynaud-Bouret2003}.
\begin{lemme}\label{inegconcentrationpois}
Let $f_1,\dots,f_n$ be $n$   bounded measurable functions.
Let $\rho, \upsilon$ be positive numbers such that  $\rho \geq \max_{1 \leq i \leq n} \|f_i\|_{\infty}$ and
$$  \frac{1}{n} \sum_{i=1}^n \int_{\TT} f_i^2 (t) s_i (t) \d \mu (t) \leq \upsilon.$$
Then, for all $r \geq 0$,
\begin{eqnarray*}
\P\left( \frac{1}{n}  \sum_{i=1}^n \left[\int_{\TT} f_i (t) \d N_i (t) - \E \left(\int_{\TT} f_i (t) \d N_i (t) \right) \right]   \geq r \right) &\leq& \exp \left( - n \frac{\upsilon }{b^2} h \left( \frac{ \rho r  }{\upsilon}\right)  \right)\\
&\leq& \exp \left( - n \frac{r^2}{2 \left(\upsilon +  \frac{\rho r}{3} \right)} \right) 
\end{eqnarray*}
where $h$ is the function defined for $u \in (-1,+\infty)$ by $h(u) = (1+u) \log(1+u)-u$.
\end{lemme}
\begin{proof}
By homogeneity we can assume that $\rho = 1$. We  assume moreover that for each  $i \in \{1,\dots, n\}$,  $f_i$ is  a piecewise  constant function (with a finite number of pieces). There exist thus $k_1,\dots , k_n \in \N^{\star}$ and a family $(a_{i,j})_{\substack{1 \leq i \leq n \\ 1 \leq j \leq k_i }}$ of elements of $[-1,1]$ such that 
 $$\forall t \in \TT , \quad f_i (t) = \sum_{j=1}^{k_i} a_{i,j} \mathbbm{1}_{A_{i,j}} (t)$$ 
 where the $A_{i,j}$ are measurable sets of $\TT$ such that $A_{i,j} \cap A_{i,j'} = \emptyset$ for all $j \neq j'$. We have for all $\xi \geq 0$
\begin{eqnarray*}
\log \E \left(e^{\xi \sum_{i=1}^n \left[ \int_{\TT} f_i \d N_i - \E \left( \int_{\TT} f_i \d N_i \right) \right] } \right) &=& \sum_{i=1}^n \log \E \left(e^{\xi \left[ \int_{\TT} f_i \d N_i - \E \left( \int_{\TT} f_i \d N_i \right) \right] } \right) \\
&=&  \sum_{i=1}^n \sum_{j=1}^{k_i} \log \E \left(e^{\xi  a_{i,j} \left[ N_i(A_{i,j}) - \E \left(  N_i(A_{i,j}) \right) \right] } \right) \\
&=& \sum_{i=1}^n \sum_{j=1}^{k_i}  \E\left( N_i(A_{i,j})\right)  ( e^{\xi a_{i,j}} - \xi a_{i,j} - 1).
\end{eqnarray*}
By using the monotony of the function $x \mapsto (e^x - x - 1)/x^2$,
\begin{eqnarray*}
\log \E \left(e^{\xi \sum_{i=1}^n \left[ \int_{\TT} f_i \d N_i - \E \left( \int_{\TT} f_i \d N_i \right) \right] } \right)
&\leq&  \sum_{i=1}^n \sum_{j=1}^{k_i}\E\left( a_{i,j}^2 N_i(A_{i,j})\right)  ( e^{\xi} - \xi  - 1) \\
&\leq& n \upsilon  ( e^{\xi} - \xi  - 1).
\end{eqnarray*}
This inequality still holds when the $f_i$ are not piecewise constant since a measurable function can be approximated by piecewise constant functions. 
Indeed,  there exists a sequence  $(f_i^{(k)})_{k \geq 1}$  of piecewise constant functions (with a finite number of jumps) such that $f_i^{(k)} \rightarrow f_i$ when $k \rightarrow +\infty$ in the space $\L^2(\TT, s_i \d \mu)$ and such that 
$\|f_i^{(k)}\|_{\infty} \leq 1$ whatever $k, i$.
By using Fatou lemma,
\begin{eqnarray*}
\log \E \left( \liminf_{k \rightarrow +\infty} e^{\xi \sum_{i=1}^n \left[ \int_{\TT} f_i^{(k)} \d N_i -  \int_{\TT} f_i^{(k)} s_i \d \mu_i  \right] } \right)
\leq n \upsilon  ( e^{\xi} - \xi  - 1).
\end{eqnarray*}
Since,
$$\E \left[ \left| \int_{\TT} f_i^{(k)} \d N_i - \int_{\TT} f_i  \d N_i \right| \right] \leq  \int_{\TT} \left|f_i^{(k)} - f_i \right| s_i \d \mu \rightarrow 0$$
one can assume (up to considering a subsequence) that $\int_{\TT} f_i^{(k)} \d N_i - \int_{\TT} f_i  \d N_i \rightarrow 0$ almost surely (for all $i \in \{1,\dots n\}$).
We then have  
\begin{eqnarray*}
\log \E \left( e^{\xi \sum_{i=1}^n \left[ \int_{\TT} f_i  \d N_i -  \int_{\TT} f_i s_i \d \mu_i  \right] } \right)
\leq n \upsilon  ( e^{\xi} - \xi  - 1)
\end{eqnarray*}
as wished.
 The concentration inequality is then deduced from the Cramér-Chernoff method, see Chapter~2 of \cite{Massart2003}.
\end{proof}

Let us return to the proof of Lemma~\ref{lemmatestrobuste}. We  define the function $\zeta$ on $[0,+\infty)^2$ by
 $$\zeta(x,y) = \frac{1}{\sqrt{2}} \left(\sqrt{\frac{y}{x + y}} -  \sqrt{\frac{x}{x + y}} \right) \quad \text{for all $x,y \in [0,+\infty)$} ,$$
where we use the convention $0/0 = 0$. Let then
\begin{eqnarray*}
Z_{f,f'} (\mathbf{N},\mathbf{x}) &=& T_{f,f'} (\mathbf{N},\mathbf{x})- \E \left[ T_{f,f'} (\mathbf{N},\mathbf{x})\right] \\
&=& \int_{\TT \times \XX} \zeta(f, f') \d M - \E \left( \int_{\TT \times \XX} \zeta(f, f') \d M \right).
\end{eqnarray*}
We  use the claim below whose proof ensues from the proofs  of Propositions~2 and 3 of~\cite{BaraudMesure}.
\begin{Claim} \label{ClaimInequaliteDansPreuvePoisson}
$$\E \left[T_{f,f'} (\mathbf{N},\mathbf{x})\right] \leq \left(1 + \frac{1}{\sqrt{2}} \right) H^2(s,f) - \left(1 - \frac{1}{\sqrt{2}} \right) H^2(s,f') $$
and
$$\frac{1}{n} \sum_{i=1}^n \int_{\TT} \zeta^2 \left(f (t, x_i), f' (t, x_i) \right)  s (t,x_i) \d \mu (t)  \leq H^2(s,f) +  H^2 (s,f') + H^2(f,f').$$
\end{Claim}
We derive from the first point of the claim that
\begin{eqnarray*}
 \P \left[ T_{f,f'} (\mathbf{N},\mathbf{x})\geq  z \right]  &=&    \P \left[ Z_{f,f'} (\mathbf{N},\mathbf{x})\geq  z - \E \left[T_{f,f'} (\mathbf{N},\mathbf{x})\right] \right] \\
  &\leq&    \P \left[ Z_{f,f'} (\mathbf{N},\mathbf{x})\geq  z - \left(1 + \frac{1}{\sqrt{2}} \right) H^2(s,f) + \left(1 - \frac{1}{\sqrt{2}} \right) H^2(s,f')   \right].
\end{eqnarray*}
Note that the random variable $Z(f,f')$ can be written as
$$Z_{f,f'} (\mathbf{N},\mathbf{x})= \frac{1}{n} \sum_{i=1}^n  \left[ \int_{\TT} \zeta \left(f (\cdot, x_i), f' (\cdot, x_i) \right) \d N_i - \E \left(\int_{\TT}\zeta \left(f (\cdot, x_i), f' (\cdot, x_i) \right) \d N_i \right) \right].$$
When $$r =   z - \left(1 + \frac{1}{\sqrt{2}} \right) H^2(s,f) + \left(1 - \frac{1}{\sqrt{2}} \right) H^2(s,f')  $$ 
is non-negative, we apply Lemma~\ref{inegconcentrationpois} with $f_i (\cdot) =  \zeta \left(f (\cdot, x_i), f' (\cdot, x_i) \right)$, $\rho = 1/\sqrt{2}$ and $$\upsilon =  H^2(s,f) +  H^2 (s,f') + H^2(f,f')$$
to obtain
\begin{eqnarray*}
\P \left[ T_{f,f'} (\mathbf{N},\mathbf{x})\geq  z \right] \leq \exp \left( - \frac{n r^2}{2 \upsilon + \frac{r \sqrt{2}}{3}} \right).
\end{eqnarray*}
We now bound from above the right-hand side of this inequality. 

For this, we begin to bound $\upsilon$ from above. We deduce from the triangular inequality and from
$4 H (s,f) \leq H (f,f')$ that
\begin{eqnarray*}
\upsilon &\leq&  3 H^2(s,f) +  3 H^2(f,f') \\
&\leq& 3 (1 + 1/16) H^2(f,f').
\end{eqnarray*} 
Now, we bound  $r$ from below.
Note that
\begin{eqnarray*}
H(f,f') \leq   H(s,f) +   H (s,f')  \leq \frac{1}{4} H (f,f') +  H (s,f')
\end{eqnarray*}
and thus $H (s,f') \geq 3/4 H(f,f').$ This leads to
\begin{eqnarray*}
r &\geq&   z - \left(1 + \frac{1}{\sqrt{2}} \right) \frac{1}{16} H^2(f,f')  + \left(1 - \frac{1}{\sqrt{2}} \right) \frac{9}{16} H^2(f,f')   \\
&\geq&  z + C H^2(f,f')  
\end{eqnarray*}
where    $C = (8 - 5 \sqrt{2})/{16}> 0$. There are two types of cases involved.
\begin{itemize}
\item If $ z + C H^2(f,f') > 0$,  $r$ is non-negative and thus
\begin{eqnarray*}
\P \left[ T_{f,f'} (\mathbf{N},\mathbf{x})\geq z \right]  \leq  \exp \left( - \frac{n \left( z + C H^2(f,f') \right)^2}{6 (1 + 1/16 ) H^2(f,f')  + \frac{\sqrt{2}}{3} \left( z + C H^2(f,f')   \right)  } \right).
\end{eqnarray*}
Set $C' = 9 \sqrt{2} (1+1/16) + C.$ Then,
\begin{eqnarray*}
\P \left[ T_{f,f'} (\mathbf{N},\mathbf{x})\geq z \right] \leq \exp \left( -  \frac{3 n}{\sqrt{2}} \frac{ \left( z + C H^2(f,f') \right)^2}{z +  C' H^2(f,f')  }\right).
\end{eqnarray*}
One can then verify that
\begin{eqnarray*}
\frac{(z + C H^2(f,f'))^2}{z + C' H^2(f,f')} \!\!\!\!\! &=& \!\!\!\!\! \frac{C^2}{C'} H^2(f,f') + \frac{(2 C' - C) C}{C'^2} z+ \frac{(C-C')^2 z^2}{C'^2 (z + C' H^2(f,f'))}  \\
&\geq& \!\!\!\!\! \frac{C^2}{C'} H^2(f,f') + \frac{(2 C' - C) C}{C'^2} z
\end{eqnarray*}
which implies that
\begin{eqnarray} \label{eqPreuvePoisson}
\P \left[ T_{f,f'} (\mathbf{N},\mathbf{x})\geq     z \right]   \leq   \exp \left( -   \frac{3 n}{\sqrt{2}}  \left( \frac{C^2}{C'} H^2(f,f') + \frac{(2 C' - C) C}{C'^2} z \right) \right).
\end{eqnarray}
\item If  $z + C H^2(f,f') \leq 0$, then
\begin{eqnarray*}
\frac{C^2}{C'} H^2(f,f') + \frac{(2 C' - C) C}{C'^2} z &\leq& \left(\frac{C^2}{C'} - \frac{(2 C' - C) C^2}{C'^2}\right) H^2 (f,f')\\
&\leq& 0.
\end{eqnarray*}
Consequently,  (\ref{eqPreuvePoisson}) also holds.
\end{itemize}
We thus have proved that
$$\P \left[ T_{f,f'} (\mathbf{N},\mathbf{x})\geq b z \right] \leq \exp \left[ -n  a  \left( H^2(f,f') + z \right) \right]$$
where
\begin{eqnarray*}
a &=& \frac{3 C^2}{\sqrt{2} C'} \simeq 4.5 \times 10^{-4} \\
 b &=& \frac{C C'}{2C' - C} \simeq 0.029.
\end{eqnarray*}
This ends the proof.
\qed

\subsubsection{Proof of Lemma~\ref{lemmaDModels}.}  \label{SectionProoflemmaDModels}
By using Proposition~7 of \cite{BirgeTEstimateurs},  we derive from $S_V(\eta)$ a set $S_V' (\eta) \subset V$
such that
$$\forall \varphi \in \LL^2 (\TT \times \XX, M), \, \forall x \geq 2, \quad |S_V' (\eta) \cap \B(\varphi, x \eta)| \leq \exp \left( 7 D_V (\eta/ 2) x^2 \right)$$
where $\B(\varphi, x \eta)$ is the ball centered at $\varphi$ with radius $x \eta$ of the metric space $(\LL^2 (\TT \times \XX, M), d_2)$, and such that
$$\forall f \in  V  , \quad d_2 (f, S_V' (\eta)) \leq \eta.$$
Proposition $12$ of \cite{BirgeTEstimateurs} (applied with $T = S_V' ( \eta)$, $M_0$ the cone of non-negative functions of $\L^2 (\TT \times \XX, M) $, $M' = \L^2 (\TT \times \XX, M)$ and $\bar{\pi}$ defined by $\bar{\pi} (f) = f \vee 0$) provides a subset $S_V'' (\eta)$ such that the functions $f \in S_V''(\eta)$ are non negative, such that
$$\forall f \in \L^2 (\TT \times \XX, M)  , \, \forall x \geq 2 , \quad |S_V'' (\eta) \cap  \B (f, x \eta) | \leq \exp \left(63 D_V (\eta) \, x^2 \right)$$
 and such that
$$\text{for all non-negative function $f \in  \L^2 (\TT \times \XX, M)$,}  \quad d_2 (f, S_V'' (\eta) ) \leq 4 d_2 (f,S_V' (\eta)).$$
The lemma holds with $\bar{S}_{V} (\eta)  =  \{\sqrt{f} , \, f \in S_V'' (\eta) \}$. \qed

\subsection{Proof of Lemma~\ref{lemmeproduitlinear}.}
The proof of this proposition requires the following elementary lemma.
\begin{lemme} \label{calculdehellinger}
Let $f, f' \in \L^2 (\TT,\mu)$ and  $g, g' \in \L^2 (\XX,\nu_n)$ such that $\|f\|_{\t} = \|f'\|_{\t} = 1 $ and $\|g\|_{\x} = \|g'\|_{\x} = 1$. Let  $\kappa, \kappa' \in \R$. Then,
$$d^2_2 \left( \kappa f  g , \kappa' f'  g' \right) = \left( \kappa - \kappa' \right)^2  + \kappa \kappa' \left( d_{\t} ^2( f , f') + d_{\x}^2( g , g') - 1/2 \, d_{\t}^2( f, f') d_{\x}^2(g , g') \right). $$
\end{lemme}
Let $\eta > 0$. In this proof, we say that a set $S(\eta)$ is a $\eta$-net of a  set $V$ in a metric space $(E,d)$ if, for all $y \in V$, there exists $x \in S(\eta)$ such that $d(x,y) \leq \eta$. 

Let us denote by $S_1$ (respectively $S_2$) the unit sphere of $V_1$ (respectively $V_2$).
Let   $S_{1} (\eta) \subset S_1$ (respectively $S_{2} (\eta) \subset S_{2}$ )  be a $\eta$-net of $S_{1}$ (respectively $S_{2}$) such that
\begin{eqnarray}
\forall f \in \L^2 (\TT,\mu), \; \forall x \geq 0, \quad  \left|S_{1} (\eta) \cap \B_{t} (f, x \eta) \right| &\leq& (2 x + 1)^{\dim V_{1}} \label{eqPreuveMajorationFonctionProduit1} \\
\forall g \in  \L^2 (\XX,\nu_n) , \; \forall x \geq 0, \quad \left|S_{2} (\eta) \cap \B_{\x} (g, x \eta) \right| &\leq& (2 x + 1)^{\dim V_{2}} \label{eqPreuveMajorationFonctionProduit2}
\end{eqnarray}
where $\B_{t} (f, x \eta) $ and $\B_{\x} (g, x \eta)$ are the closed balls centered at $f$ and $g$ with   radius $x \eta$  of the metric spaces $(\L^2 (\TT,\mu), d_{\t}) $ and $(\L^2 (\XX,\nu_n), d_{\x})$ respectively.
We refer to Lemma~4 of \cite{BirgeTEstimateurs} for the existence of these nets.
Let now 
$$S (\eta) = \bigcup_{k \in \N^{\star}}  \left\{ \frac{k \eta}{\sqrt{2}}   f g, \, (f,g) \in   S_{1} \left( \frac{1}{ \sqrt{2} k }  \right) \times S_{2}\left( \frac{1}{ \sqrt{2} k } \right) \right\}. $$
First of all, $S (\eta)$ is a $\eta$-net of $V$. Indeed, let $\varphi \in V$. We can write $\varphi (t,x) = \kappa f (t) g (x)$ where $\kappa \geq 0$, $f \in S_1$ and $g \in S_{2}$.
Let us define $$k =  \inf \left\{i \in \N^{\star}, \, i \geq  \sqrt{2} \kappa / \eta \right\} $$ and  let $(f' ,g') \in S_{1} (1/(\sqrt{2} k)) \times S_{2} (1/(\sqrt{2} k)) $ such that $$d_{\t} (f , f') \leq\frac{1}{ \sqrt{2} k } \quad \text{and} \quad d_{\x} (g, g') \leq \frac{1}{ \sqrt{2} k }.$$
By using Lemma \ref{calculdehellinger}, the application $\varphi' (t,x) = \frac{k \eta}{\sqrt{2} }  f' (t) g' (x)$ is such that 
$$d_2 \left( \varphi , \varphi' \right) \leq  \eta,$$
which proves that $S (\eta)$ is a $\eta$-net of $V$.

According to  Definition~\ref{definitionmetricspace},  we now consider $\varphi  \in \L^2 (\TT \times \XX, M)$  and $x \geq 2$ and aim at bounding from above the cardinality of the set $S (\eta) \cap \B \left( \varphi , x \eta \right)$ (where we recall that $ \B \left(\varphi , x \eta \right)$ is the closed ball centered at $\varphi$ with radius $x \eta$ of the metric space $(\LL^2 (\TT \times \XX, M),d_2)$). 

For this purpose, we begin to assume that  $\varphi$ belongs to $S(\eta)$, which implies that the function can be written as $\varphi (t,x) = \kappa f (t) g (x) $. We introduce
\begin{eqnarray*}
\mathcal{K} = \left\{\frac{k \eta}{\sqrt{2}} ,\, k \in \N^{\star}, \, \left|\frac{k \eta}{\sqrt{2}} - \kappa\right| \leq x \eta \right\} 
\end{eqnarray*}
and for all $\kappa' \in \mathcal{K}$,
\begin{eqnarray*}
\mathcal{C}  (\kappa') = \left(S_{1} \left( \frac{\eta}{2 \kappa'} \right) \cap \B_{\t} \left(f, 6 x^{2} \frac{\eta}{2 \kappa'} \right) \right) \times \left(S_{2} \left( \frac{\eta}{2 \kappa'} \right) \cap \B_{\x} \left(g, 6 x^{2}  \frac{\eta}{2 \kappa'} \right) \right).
\end{eqnarray*}
Let $$T(\eta) =  \left\{\kappa' f' g' ,\; \kappa' \in \mathcal{K}, \; (f',g') \in \mathcal{C} (\kappa')  \right).$$
We shall prove that
\begin{eqnarray} \label{eqInclusionPreuveFonctionProduit}
S (\eta) \cap \B \left( \kappa f g , x \eta \right) \subset T(\eta).
\end{eqnarray}
We then upper-bound the cardinality of $S (\eta) \cap \B \left( \kappa f g , x \eta \right)$ by bounding from above the cardinality of $T(\eta)$

Let  $\varphi'  \in S (\eta) \cap \B \left( \kappa f g , x \eta \right)$.  There exist $\kappa'$, $f'$ and $g'$ such that $\varphi' =  \kappa' f'  g'$ and we derive from Lemma~\ref{calculdehellinger} that
$$(\kappa - \kappa')^2 \leq  d_2^2 \left(\varphi , \varphi'  \right)  \leq x^2 \eta^2, $$
which implies that $\kappa' \in \mathcal{K}$. 
We now distinguish several cases.

\begin{itemize}
\item Suppose that  $$\left(\int_{\TT} f(t) f' (t) \d \mu(t)\right) \left(\int_{\XX} g (x) g'(x) \d \nu_n(x) \right) < 0.$$
We then have  $d_2^2 (\varphi , \varphi') \geq \kappa^2 + \kappa'^2.$  Since $\kappa \geq \eta / \sqrt{2}$, $\kappa' \leq \kappa + x \eta$ and  $x \geq 2$, 
$$\frac{\kappa'}{\kappa} \leq 1 + \sqrt{2} x \leq \frac{3}{2} x.$$ 
Thus, $d_2^2 (\varphi , \varphi') \geq  4  \kappa'^2 / (9 x^2)$.  Since $f,f' \in S_1$ and $g,g' \in S_2$,
\begin{eqnarray*}
\|f - f'\|_{\t}^2 \leq 4 \quad \text{and} \quad    \|g-g'\|_{\x}^2  \leq 4   \
 \end{eqnarray*}
and thus 
\begin{eqnarray*}
\|f - f'\|_{\t}^2 \leq \frac{9 x^2}{\kappa'^2} d_2^2(\varphi,\varphi') \leq  \frac{9 x^4}{\kappa'^2} \eta^2  \\
   \|g-g'\|_{\x}^2   \leq \frac{9 x^2}{\kappa'^2} d_2^2(\varphi,\varphi') \leq  \frac{9 x^4}{\kappa'^2} \eta^2. 
\end{eqnarray*}
We then have
 $$f' \in S_{1} \left( \frac{\eta}{2 \kappa'} \right) \cap \B_{\t} \left(f, 6 x^{2} \frac{\eta}{2 \kappa'} \right) \quad \text{and} \quad g' \in S_{2} \left( \frac{\eta}{2 \kappa'} \right) \cap \B_{\x} \left(g, 6 x^{2} \frac{\eta}{2 \kappa'} \right)$$
that is $(f',g') \in \mathcal{C} (\kappa') $ and thus $\varphi' \in T(\eta)$.
\item If now, $$\int_{\TT} f(t) f' (t) \d \mu(t) > 0 \quad \text{and} \quad \int_{\XX} g (x) g'(x) \d \nu_n(x)  > 0 ,$$
then $d_{\t}^2 (f , f' ) \leq 2$ and $d_{\x}^2 (g , g') \leq 2$. We then  derive from Lemma~\ref{calculdehellinger} and from the elementary inequality 
 $$\frac{1}{2} (y_1 + y_2) \leq y_1 + y_2 - \frac{1}{2}  y_1 y_2 \quad \text{for all $y_1,y_2\in [0,2]$,}$$
 that
$$(\kappa - \kappa')^2 + \frac{\kappa \kappa'}{2} \left(d_{\t}^2 (f , f' ) + d_{\x}^2 (g , g'
) \right) \leq   d_2^2 \left( \varphi , \varphi' \right)  \leq x^2 \eta^2.$$
Hence,
 $$ d_{\t}^2 (f , f' ) + d_{\x}^2 (g , g')  \leq  \frac{2 x^2 \eta^2}{\kappa \kappa'} .$$
 By using the inequality $\kappa' / \kappa \leq 3/2 x$ proved in the first point, we  deduce
$$f' \in S_{1} \left( \frac{\eta}{2 \kappa'} \right) \cap \B_{\t} \left(f, 2 \sqrt{3} x^{3/2} \frac{\eta}{2 \kappa'} \right) \quad \text{and} \quad g' \in S_{2} \left( \frac{\eta}{2 \kappa'} \right) \cap \B_{\x} \left(g, 2 \sqrt{3} x^{3/2} \frac{\eta}{2 \kappa'} \right).$$
Since $2 \sqrt{3} x^{3/2} \leq 6 x^2$ (because $x \geq 2$), we have $(f',g') \in \mathcal{C} (\kappa')$ and thus $\varphi' \in T(\eta)$.
\item Finally, assume that  $$\int_{\TT} f(t) f' (t) \d \mu(t) < 0 \quad \text{and} \quad \int_{\XX} g (x) g'(x) \d \nu_n(x)  < 0.$$
Note that the function $\varphi'$ can also be written as $\varphi' = \kappa' (-f') (-g').$ We then deduce from the second point that
$$-f' \in S_{1} \left( \frac{\eta}{2 \kappa'} \right) \cap \B_{\t} \left(f, 2 \sqrt{3} x^{3/2} \frac{\eta}{2 \kappa'} \right) \; \text{and} \; -g' \in S_{2} \left( \frac{\eta}{2 \kappa'} \right) \cap \B_{\x} \left(g, 2 \sqrt{3} x^{3/2} \frac{\eta}{2 \kappa'} \right)$$
and thus $(-f',-g') \in \mathcal{C} (\kappa')$. Hence,  $\varphi' \in T(\eta)$ as wished.
\end{itemize}
We thus have proved (\ref{eqInclusionPreuveFonctionProduit}) and  thus
\begin{eqnarray*}
\left|S (\eta) \cap \B \left( \kappa f g , x \eta \right) \right| \leq   \sum_{\kappa' \in  \mathcal{K}}   \left| \mathcal{C} (\kappa')\right| .
\end{eqnarray*}
Now, note that $|\mathcal{K} | \leq 2 \sqrt{2} x + 1$. By using (\ref{eqPreuveMajorationFonctionProduit1}) and (\ref{eqPreuveMajorationFonctionProduit2}),   for all $\kappa'$,
$$|\mathcal{C} (\kappa')| \leq \left(12 x^2 + 1 \right)^{\dim V_{1} + \dim V_{2}}.$$
Consequently, we have proved 
$$\forall \varphi  \in S(\eta), \, \forall x \geq 2 , \quad \left|S (\eta) \cap \B \left( \varphi , x \eta \right)\right| \leq  \big(2\sqrt{2} x + 1 \big) \left(12 x^2 + 1 \right)^{\dim V_{1} + \dim V_{2}}.$$

Let us recall that we must to upper bound the cardinality of $S (\eta) \cap \B \left( \varphi , x \eta \right)$ for all  $\varphi  \in \L^2 (\TT \times \XX, M)$. 
For this, if  $\varphi \in \L^2 (\TT \times \XX, M)$, may be $ \left|S (\eta) \cap \B \left( \varphi , x \eta \right)\right|  = 0$. If not,
there exists $\varphi' \in S (\eta) \cap \B \left( \varphi , x \eta \right)$ and thus
$$\left|S (\eta) \cap \B \left( \varphi , x \eta \right)\right| \leq  \left|S (\eta) \cap \B \left( \varphi' , 2 x \eta \right)\right|.$$
Consequently, for all $\varphi  \in \L^2 (\TT \times \XX, M)$,
$$\forall x \geq 2 , \quad \left|S (\eta) \cap \B \left( \varphi , x \eta \right)\right| \leq \big(4 \sqrt{2} x + 1 \big) \left(48 x^2 + 1 \right)^{\dim V_{1} + \dim V_{2}}.$$
The conclusion ensues from the elementary inequalities 
$$\forall x \geq 2 , \quad  4 \sqrt{2} x + 1  \leq e^{1.4 x^2}  \; \text{and} \;  48 x^2 + 1  \leq e^{1.4 x^2}. $$ \carre

\subsection{Proof of Proposition~\ref{propmultiplicativemodels}.}
For all pair $(V_{1}, V_{2}) \in \V_{1} \times \V_{2}$,  we define the set~$V$ by relation~{(\ref{Vmultiplicatif})}.
Let then $\V$ be the collection of all $V$ when $(V_{1}, V_{2})$ varies among $\V_{1} \times \V_{2}$.
Let $\bar{\Delta}$ be the application on $\V$ defined by
$$\bar{\Delta} \left( V \right) = \Delta_1(V_{1}) + \Delta_2(V_{2})$$
when $V$ corresponds to $(V_{1}, V_{2})$.
We apply afterwards Theorem~\ref{thmselectionmodelegeneral} with $\V$ and $\bar{\Delta}$ to derive an estimator $\hat{s}$ such that
\begin{eqnarray*}
C \E\left[H^2(s, \hat{s}) \right]  \leq   \inf_{V \in \V} \left\lbrace d^2_2 \left(\sqrt{s},V  \right) +   \frac{ \dim V_{1} + \dim V_{2} + 1 + \Delta_1 (V_{1})+ \Delta_2 (V_{2})}{n} \right\rbrace .
\end{eqnarray*}
Let  $\kappa v_1 v_2 \in \F$, and let $(v'_1,v'_2) \in V_{1} \times V_{2}$ such that $\|v'_1\|_{\t} = \|v'_2\|_{\x} = 1$. The preceding inequality implies
\begin{eqnarray*}
C' \E\left[H^2(s, \hat{s}) \right]  &\leq&  d_2^2 \left( \sqrt{s} , \kappa v_1 v_2 \right) +   \kappa^2 d^2_2 \left( v_1 v_2  ,  v_1' v_2'   \right) \\
& & \quad +   \frac{ \dim V_{1} \vee 1 + \dim V_{2}  \vee 1 + \Delta_1 (V_{1})+ \Delta_2 (V_{2})}{n} 
\end{eqnarray*}
where $C' = C/2$. By using Lemma~\ref{calculdehellinger} (page \pageref{calculdehellinger}),
$$ d^2_2 \left( v_1 v_2  ,  v_1' v_2'   \right) \leq  d_{\t}^2 \left(v_1 , v_1' \right) + d_{\x}^2 \left(v_2 , v_2' \right).$$
By taking the infimum over all $v_1'$ and $v_2'$,
\begin{eqnarray*}
C' \E\left[H^2(s, \hat{s}) \right]  &\leq&   d_2^2 \left( \sqrt{s} , \kappa v_1 v_2 \right) + \kappa^2 d^2_{\t} \left( v_1  , S_{1}  \right) +   \frac{ \dim V_{1} \vee 1 + \Delta_1 (V_{1})}{n}   \\
& & \quad  +  \kappa^2 d^2_{\x} \left( v_2 , S_{2}  \right) + \frac{\dim V_{2}  \vee 1+ \Delta_2 (V_{2})}{n}
\end{eqnarray*}
where $S_{1}$ and $S_{2}$ are the unit spheres of $V_{1}$ and $V_{2}$ respectively. 

Now, remark that  
$$d_{\t} (v_1, S_{1})  \leq 2 d_{\t} (v_1, V_{1}) \quad \text{and} \quad d_{\x} (v_2, S_{2})  \leq 2 d_{\x} (v_2, V_{2}).$$
Indeed, if $w_1$ is the projection of $v_1$ on  $V_1$, then
\begin{eqnarray*}
d_{\t} (v_1, S_{1}) &\leq& \left\|v_1 - \frac{w_1}{\|w_1\|_{\t}} \right\|_{\t} \\
&\leq& \left\|v_1 - w_1 \right\|_{\t} + \left\|w_1 - \frac{w_1}{\|w_1\|_{\t}} \right\|_{\t} .
\end{eqnarray*}
Now,
$$ \left\|w_1 - \frac{w_1}{\|w_1\|_{\t}} \right\|_{\t} = \left|1 - \frac{1}{\|w_1\|_{\t}} \right| \left\|w_1 \right\|_{\t} = \left|  \left\|w_1 \right\|_{\t} - 1 \right|.$$
Since  $\|v_1\|_{\t} = 1$,
$$ \left\|w_1 - \frac{w_1}{\|w_1\|_{\t}} \right\|_{\t} =  \left|  \left\|w_1 \right\|_{\t} - \|v_1\|_{\t} \right| \leq \left\|v_1 - w_1 \right\|_{\t} .$$
Since $\left\|v_1 - w_1 \right\|_{\t}  = d_{\t} (v_1, V_{1}) $, we have $d_{\t} (v_1, S_{1})  \leq 2 d_{\t} (v_1, V_{1})$. The proof of the inequality $ d_{\x} (v_2, S_{2})  \leq 2 d_{\x} (v_2, V_{2})$ is analogue.

The conclusion follows.
\carre

\subsection{Proof of Lemma~\ref{lemmepowerlawparametric1}.}
For all $b,b'  \in (-1/2, +\infty)$,
\begin{eqnarray*}
\int_0^1 \left( \sqrt{2 b + 1} t^{b} -  \sqrt{2 b' + 1} t^{b'}  \right)^2 \d t =  \frac{4 \left(b - b' \right)^2}{(1 + b + b') \left(\sqrt{2 b + 1} + \sqrt{2 b'+1} \right)^2}
\end{eqnarray*} 
and thus 
\begin{eqnarray*}
\frac{ \left(b - b' \right)^2}{(1 + 2 (b \vee b'))^2 } \leq \int_0^1 \left( \sqrt{2 b + 1} t^{b } -  \sqrt{2 b' + 1} t^{b'}  \right)^2 \d t \leq  \frac{ \left(b  - b' \right)^2}{ (1 + 2 (b \wedge b'))^2 }
\end{eqnarray*} 
which ends the proof. \qed
\subsection{Proof of Lemma~\ref{lemmepowerlawparametric2}.}
For $b > 0$, we define
 $$g_b(t) = \frac{f_b (t)}{\|f_b\|_{\t}} = \frac{2^{k/2} b^{1/2+k/2}}{\sqrt{k/2 \, k!}} \, t^{k/2} e^{-b t}.$$
For all $b,b' > 0$,
\begin{eqnarray*}
\frac{1 }{2 }\int_0^{\infty} \left( g_{b } (t) - g_{b'} (t)  \right)^2 \d t &=&  1 - \frac{\left(2 \sqrt{b  b'} \right)^{k + 1}}{(b +b')^{k + 1}}  \\
&=& \frac{\sum_{j=0}^{k} (b +b')^{k-j} (2 \sqrt{b  b'})^{j}}{(b +b')^{k + 1} (\sqrt{b }+ \sqrt{b'})^2} \left(b  - b' \right)^2 \\
&=& \frac{1}{(b +b') (\sqrt{b }+ \sqrt{b'})^2} \sum_{j=0}^{k} \left( \frac{2 \sqrt{b  b'}}{b  + b'} \right)^j  \left(b  - b' \right)^2.
\end{eqnarray*} 
Consequently, 
\begin{eqnarray*}
\frac{1}{8 (b \vee b')^2} \left(b  - b' \right)^2 \leq \frac{1 }{2 }\int_0^{\infty} \left( g_{b } (t) - g_{b'} (t)  \right)^2 \d t &\leq&  \frac{k + 1}{8 (b \wedge b')^{2}} \left(b  - b' \right)^2
\end{eqnarray*} 
which concludes the proof.
\qed

\subsection{Proof of Proposition~\ref{proppowerlawundercox2}.}
We generalize Lemma~\ref{lemmeproduitlinear} for some new spaces. The proof of the following lemma  is analogue to the one of Lemma~\ref{lemmeproduitlinear} and will not be detailed.

\begin{lemme} \label{lemmeproduitparametrised}
Let  $V_{1}$ and $V_{2}$ be  subsets of the unit spheres of $\L^2 (\TT, \mu)$ and $\L^2 (\XX, \nu_n)$ respectively. 
For each $i \in \{1,2\}$, we assume that there exist  positive numbers $\underline{\rho_i}, \bar{\rho_i}$,  a subset  $W_{\mathtt{i}}$ of a finite dimensional normed linear space $(\bar{W}_{\mathtt{i}},|\cdot|_{\mathtt{i}})$ and  a  surjective map $\Phi_i$ from~$W_{\mathtt{i}}$ onto~$V_i$ such that:
\begin{eqnarray} \label{conditionlemmeproduitparametrized1}
\forall (x,y) \in W_{\mathtt{1}} , \quad \underline{\rho_1} |x - y |_1 \leq d_{\t} (\Phi_1 (x) , \Phi_1 (y) ) \leq \bar{\rho}_1  |x - y |_{\mathtt{1}} 
\end{eqnarray}
\begin{eqnarray} \label{conditionlemmeproduitparametrized2}
\forall (x,y) \in W_{\mathtt{2}} , \quad \underline{\rho_2} |x - y |_2 \leq d_{\x} (\Phi_2 (x) , \Phi_2 (y) ) \leq \bar{\rho}_2  |x - y |_{\mathtt{2}}.
\end{eqnarray}
The set 
\begin{eqnarray*} 
V =  \left\{\kappa  v_1 v_2  , \, (v_1,v_2) \in V_{1} \times V_{2} ,\, \kappa  \in [0, +\infty) \right\}
\end{eqnarray*}
has a finite metric dimension bounded by
$$D_{V} = C \left[1 +  \log \left(1 + \frac{\bar{\rho}_1}{\underline{\rho_1}} \right) \dim {\bar{W}_{\mathtt{1}}} + \log \left(1 + \frac{\bar{\rho}_2}{\underline{\rho_2}} \right) \dim \bar{W}_{\mathtt{2}}   \right]$$
where $C$ is an universal constant.
\end{lemme} 
\begin{lemme}
Let for all   $r, R \in (b_0, +\infty)$, such that $R > r$,  $V_{\t} (r,R)$ be the set defined by
$$V_{\t} (r,R) =  \left\{  \frac{u_{b} }{\|u_{b}\|_{\t}} , \,   b \in [r,R]  \right\}.$$
Condition~{(\ref{conditionlemmeproduitparametrized1})} holds with $\dim {\bar{W}_{\mathtt{1}}} = 1$,  $\underline{\rho_1} = \underline{\rho} (R)$ and $\bar{\rho_1} =\bar{\rho} (r)$.
\end{lemme}

\begin{lemme} \label{lemmapreuveproduitproof}
For all positive number $\rho$ and  $W \in \W$, let
$$V_{2} (W, \rho) = \left\{  \frac{v_{\ttheta}}{ \|v_{\ttheta}\|_{\x}} , \, \ttheta \in W , \, \|\ttheta\| \leq \rho \right\}.$$
There exists a finite dimensional normed linear space  $(\bar{W}_{\mathtt{2}}, |\cdot|_2)$ and  a map $\Phi_2$ from $\bar{W}_{\mathtt{2}}$ onto $V_{2} (W,\rho)$ such that condition~{(\ref{conditionlemmeproduitparametrized2})} holds with $\dim {\bar{W}_{\mathtt{2}}} \leq \dim W$,  $\underline{\rho_2} = e^{-3 \rho}$ and $\bar{\rho}_2 = e^{3 \rho}$.
\end{lemme}
\begin{proof} [Proof of Lemma~\ref{lemmapreuveproduitproof}]
For any integers $i,j \in \N^{\star}$, let us denote by $\varphi_{i,j}$ the linear form on $\R^{k_2}$ defined by~${\varphi_{i,j} (\ttheta) = <x_i - x_j, \ttheta >}$ where $< \cdot , \cdot > $ is the standard scalar product on~$\R^{k_2}$.
 Let $W_{\mathtt{1}} = \cap_{i \neq j} \Ker \varphi_{i,j}$ and let $W_{\mathtt{2}} $ such that
$W = W_{\mathtt{1}}  \oplus W_{\mathtt{2}} $ and such that $<u,v> = 0$ for all $(u,v) \in W_{\mathtt{1}} \times W_{\mathtt{2}}$. Since the functions of $\L^2 (\XX,\nu_n)$ are defined $\nu_n$-almost everywhere, the set $V_2 (W,\rho)$ can be written as 
$$V_2 (W, \rho) = \Phi_2 \left( \left\{\ttheta \in W_{\mathtt{2}}  , \, \|\ttheta\| \leq \rho \right\} \right) \quad \text{where} \quad \Phi_2 (\ttheta) =  \frac{v_{\ttheta}}{ \|v_{\ttheta}\|_{\x}}. $$ 
Indeed, let $\ttheta \in W$ written as $\ttheta = \ttheta_1 + \ttheta_2$ where $\ttheta_1 \in W_{\mathtt{1}} $ and $\ttheta_2 \in W_{\mathtt{2}}$. Then, for all $j \in \{1,\dots,n\}$,
\begin{eqnarray*}
\frac{v_{\ttheta} (x_j)}{ \|v_{\ttheta}\|_{\x}} &=& \frac{\exp \left(<x_j, \ttheta>\right)}{\sqrt{\sum_{i=1}^n \exp \left(2 <x_i, \ttheta> \right) }} \\
&=& \frac{\exp \left(<x_j, \ttheta_1> + <x_j, \ttheta_2>\right) }{\sqrt{\sum_{i=1}^n \exp \left(2 <x_i, \ttheta_1> + 2 <x_j, \ttheta_2> \right) }}  \\
&=& \frac{v_{\ttheta_1} (x_j)}{ \|v_{\ttheta_1}\|_{\x}}
\end{eqnarray*}
and thus $\Phi_2(\ttheta) = \Phi_2 (\ttheta_1)$, $\nu_n$-almost everywhere.

For all $x \in \XX$, let $\Psi_{x}$ be the function defined from $\XX$ into $\R$ by $\Psi_{x} (\ttheta) = \Phi_2 (\ttheta) (x) = v_{\ttheta} (x) /  \|v_{\ttheta}\|_{\x}$.
We derive from some calculus that the differential of $\Psi_x$ at the point $\ttheta \in W_2$, denoted by $d \Psi_x (\ttheta)$, is
$$\forall h \in \R^{k_2}, \quad d \Psi_x (\ttheta) \cdot h = \frac{ \frac{1}{n} \sum_{i=1}^n \exp \left(2 <\ttheta, x_i > + <\ttheta, x > \right) \left( <x - x_i, h> \right) }{\left( \frac{1}{n} \sum_{i=1}^n \exp \left(2 <\ttheta, x_i > \right)\right)^{3/2}}.$$
In particular, we have
$$\forall h \in \R^{k_2}, \quad \frac{e^{-3 \rho}}{n} \sum_{i=1}^n |<x - x_i, h >| \leq |d  \Psi_x (\ttheta) \cdot h | \leq \frac{e^{3 \rho}}{n} \sum_{i=1}^n |<x - x_i, h >|.$$
If we endow $W_{\mathtt{2}}$ with  the norm $|\cdot|_{\mathtt{2}}$ defined by 
$$\forall \ttheta \in W_{\mathtt{2}} , \quad |\ttheta|_{\mathtt{2}} = \sqrt{\frac{1}{n}  \sum_{i=1}^n \left( \frac{1}{n} \sum_{ j=1}^n |<x_i-x_j, \ttheta>| \right)^2}, $$
 the mean value theorem leads to
$$\forall (\ttheta_1,\ttheta_2) \in W_2 , \quad e^{-3 \rho} |\ttheta_1 - \ttheta_2|_{\mathtt{2}} \leq d_{\x} \left(\Phi_2 (\ttheta_1) , \Phi_2 (\ttheta_2) \right)   \leq e^{3 \rho} | \ttheta_1 - \ttheta_2|_{\mathtt{2}}, $$
which concludes the proof.
\end{proof}
We now prove Proposition~\ref{proppowerlawundercox2}. We derive from Lemma~\ref{lemmeproduitparametrised} that for all $\varrho \in (0,+\infty)$, all $r, R \in (b_0, +\infty)$ and  all $W \in \W$, the set 
$$V (r,R,W,\varrho ) = \left\{a  u_b v_{\ttheta} , \, a \in [0,+\infty), \, b \in [r,R], \, \ttheta' \in W,\, \|\ttheta'\| \leq \varrho \right\}$$
has a metric dimension bounded by $$C D_{V (r,R,W,\varrho )} =  1 + \varrho \dim W  +   \log \left(1 + \frac{\bar{\rho} (r) }{\underline{\rho} (R)} \right) $$
for some universal positive constant $C$.

Let us define the collection $\V$ by $$\V = \left\{ V \left(b_0 + 1/r , b_0 + R ,W,\varrho \right)  , \, W \in \W , \, r, R, \varrho \in \N^{\star} \right\}$$ and the map $\bar{\Delta}$ on $\V$ by $$\bar{\Delta}\left(  V (r,R,W,\varrho ) \right) =  \Delta(W) +  \log (2 R^2) + \log (2 r^2) + \log (2 \varrho^2) . $$
We apply Theorem~\ref{thmselectionmodelegeneral} with $(\V,\bar{\Delta})$ to build an estimator $\hat{s}$. For all $W \in \W$,  $\varrho, r, R \in \N^{\star}$,  $\ttheta' \in W$ such that $\|\ttheta'\| \leq \varrho$,  $a \in [0,+\infty)$,  $b \in [ b_0+1/r, b_0 + R]$, this estimator satisfies
\begin{eqnarray*}
C' \E\left[ H^2(s,\hat{s}) \right] &\leq&     d_2^2( \sqrt{s} , a  u_b v_{\ttheta'} )  \\
& &   +  \frac{1+  \varrho \dim W   + \log \left(1   + \frac{\bar{\rho} \left(b_0  + 1/r \right)}{\underline{\rho} (b_0 + R) } \right) + \Delta (W) + \log r + \log R  + \log \varrho }{n}   
\end{eqnarray*} 
where $C'$ is another  universal positive constant.

We may roughly upper-bound the right-hand side of this inequality to get
\begin{eqnarray*}
C'' \E\left[ H^2(s,\hat{s}) \right] &\leq&     d_2^2( \sqrt{s} , a  u_b v_{\ttheta'} )  \\
& &   +  \frac{1+  \varrho \dim W   + \log \left(1   \vee {\bar{\rho} \left(b_0  + 1/r \right)}\right) + \left|\log \left(1 \wedge {\underline{\rho} (b_0 + R) } \right) \right| + \Delta (W)}{n} \\
& & + \frac{\log r + \log R  + \log \varrho }{n}   
\end{eqnarray*} 
for some universal positive constant $C''$.

In particular, for all $W \in \W$, $\ttheta' \in W$, $a \in [0,+\infty)$ and $b \in I$, we may use this inequality with  $R = \inf \{i \in \N^{\star}, \, i \geq  b - b_0 \} $, $r =  \inf \{i \in \N^{\star}, i \geq 1/ (b - b_0) \}$, $\varrho = \inf \{i \in \N^{\star}, i \geq \|\ttheta'\| \}$ to derive
\begin{eqnarray*}
 C''' \E\left[ H^2(s,\hat{s}) \right] &\leq&     d_2^2( \sqrt{s} , a  u_b v_{\ttheta'} )    +  \frac{(1 \vee  \|\ttheta'\|) (1 \vee \dim W ) + \Delta(W)}{n} \\
& & + \frac{1}{n}  \left\{ \log \left[1 \vee \bar{\rho} \left(b_0 + \frac{b-b_0}{b-b_0+1} \right) \right] + \left|\log \left(1 \wedge \underline{\rho} (1+b) \right) \right| +  \left| \log (b - b_0) \right| \right\}
\end{eqnarray*} 
where $C'''$ is an universal positive constant. 

Now, by using the triangular inequality, we have for all  $\ttheta \in \R^{k_2}$, 
\begin{eqnarray*}
 d_2^2( \sqrt{s} , a  u_b v_{\ttheta'} )  &\leq&  2 \left(   d_2^2( \sqrt{s} , a  u_b v_{\ttheta} ) +  d_2^2(a u_b v_{\ttheta} , a  u_b v_{\ttheta'} \right)  \\
 &\leq&  2 \left(   d_2^2( \sqrt{s} , a  u_b v_{\ttheta} ) + a^2 \|u_b\|_{\t}^2  d_{\x}^2(v_{\ttheta} , v_{\ttheta'}) \right).
\end{eqnarray*} 
Some calculus shows that $d_{\x} (v_{\ttheta} , v_{\ttheta'} ) \leq e^{\|\ttheta\| \vee \|\ttheta'\|} \|\ttheta - \ttheta' \| $. 

Consequently, for all $a \in [0,+\infty)$, $b \in I$, $\ttheta \in \R^k$, $W \in \W$, we obtain (by taking  $\ttheta'$  the projection of~$\ttheta$ on $W$),
\begin{eqnarray*}
C'''' \E\left[ H^2(s,\hat{s}) \right] &\leq&   d_2^2 (\sqrt{s},   a  u_b v_{\ttheta} )  +      a^2 \|u_b\|_{\t}^2 e^{2 \|\theta\|} d^2 (\ttheta,W)  +  \frac{(1 \vee \dim W) (1 \vee \|\ttheta\|)   + \Delta(W) }{n}  \\
& & + \frac{1}{n}  \left\{ \log \left[1 \vee \bar{\rho} \left(b_0 + \frac{b-b_0}{b-b_0+1} \right) \right] + \left|\log \left(1 \wedge \underline{\rho} (1+b) \right) \right| +  \left| \log (b - b_0) \right| \right\} 
\end{eqnarray*}
where $C''''$ is an universal positive constant. 
We conclude by taking the infimum over all $W \in \W$. \carre

\subsection{Proofs of Lemmas~\ref{LemmeDuaneModelParametric} and \ref{LemmeDuaneModelParametric2}.}
\begin{proof} [Proof of Lemma~\ref{LemmeDuaneModelParametric}]
We derive from some calculus that for all $\theta_2, \theta_2' \in [-1/2 + 1/r_2,+\infty)$,
\begin{eqnarray*}
\int_{0}^1 (t^{\theta_2} - t^{\theta_2'})^2 \d t = \frac{2 (\theta_2-\theta_2')^2}{(1+2 \theta_2) (1+\theta_2+\theta_2') (1+2\theta_2')}  \leq 2 r_2^3 (\theta_2-\theta_2')^2.
\end{eqnarray*}
Hence, for all $(\theta_1,\theta_2), (\theta_1',\theta_2') \in  [-r_1,r_1] \times  [-1/2 + 1/r_2,+\infty)$, 
\begin{eqnarray*}
\sqrt{\int_{0}^1 (\theta_1 t^{\theta_2} - \theta_1' t^{\theta_2'})^2 \d t} &\leq& 
\sqrt{\int_{0}^1 (\theta_1  - \theta_1' )^2  t^{2 \theta_2'} \d t} + \sqrt{\int_{0}^1 \theta_1^2 (t^{\theta_2} -  t^{\theta_2'})^2 \d t} \\
 &\leq&  \frac{|\theta_1-\theta_1'|}{\sqrt{2 \theta_2' + 1}} +  \frac{\sqrt{2} |\theta_1| |\theta_2-\theta_2'|}{\sqrt{(1+2 \theta_2) (1+\theta_2+\theta_2') (1+2 \theta_2')}} \\
&\leq& r_2^{1/2} |\theta_1- \theta_1'| + \sqrt{2} r_1 r_2^{3/2} |\theta_2-\theta_2'|.
\end{eqnarray*}
This ends the proof.
\end{proof}
\begin{proof} [Proof of Lemma~\ref{LemmeDuaneModelParametric2}]
We derive from some calculus that for all   $\theta_2,\theta_2' \geq 1/r_2$,
\begin{eqnarray*}
\int_{0}^{\infty} \left( t^{k/2} e^{- \theta_2 t} - { t^{k/2}e^{-\theta_2' t}} \right)^2 \d t = \frac{k!}{2^{k+1}} \left(\frac{1}{\theta_2^k} + \frac{1}{\theta_2'^k} - \frac{2^{2+k}}{\left(\theta_2 + \theta_2'\right)^k}  \right).
\end{eqnarray*}
If $k = 0$,
\begin{eqnarray*}
\int_{0}^{\infty} \left( t^{k/2} e^{- \theta_2 t} - { t^{k/2}e^{-\theta_2' t}} \right)^2 \d t =  \frac{(\theta_2'-\theta_2)^2}{2 \theta_2'^2 \theta_2 + 2 \theta_2' \theta_2^2}  \leq \frac{r_2^3}{4} (\theta_2'-\theta_2)^2
\end{eqnarray*}
while if  $k = 1$,
\begin{eqnarray*}
\int_{0}^{\infty} \left( t^{k/2} e^{- \theta_2 t} - { t^{k/2}e^{-\theta_2' t}} \right)^2 \d t &=&  \frac{\theta_2'^2 + 4 \theta_2' \theta_2 + \theta_2^2}{4 \theta_2'^2 \theta_2^2 (\theta_2'+\theta_2)^2 } (\theta_2'-\theta_2)^2 \\
&\leq& \frac{3 r_2^4}{8}  (\theta_2'-\theta_2)^2.
\end{eqnarray*}
Hence, for all $(\theta_1,\theta_2), (\theta_1',\theta_2') \in  [-r_1,r_1] \times  [ 1/r_2,+\infty)$, 
\begin{eqnarray*}
\sqrt{\int_{0}^{\infty} \left(\theta_1 t^{k/2} e^{- \theta_2 t} - \theta_1' t^{k/2} e^{- \theta_2' t} \right)^2 \d t} &\leq& 
\sqrt{\int_{0}^{\infty}  (\theta_1  - \theta_1' )^2  t^k e^{- 2 \theta_2' t} \d t} \\
& & \qquad + |\theta_1| \sqrt{\int_{0}^{\infty} \left( t^{k/2} e^{- \theta_2 t} - { t^{k/2}e^{-\theta_2' t}} \right)^2 \d t} \\
\end{eqnarray*}
Now,
$$\sqrt{\int_{0}^{\infty}  (\theta_1  - \theta_1' )^2  t^k e^{- 2 \theta_2' t} \d t}  =   \frac{\sqrt{k!} |\theta_1-\theta_1'|}{ 2^{(k+1)/2} \theta_2'^{(k+1)/2} },$$
which ends the proof.
\end{proof}

\subsection{Proof of Theorem~\ref{thmestimationrobuste}.}
We start with the following proposition.
\begin{prop} \label{propdimensionmetricrobuste}
Suppose that Assumption~\ref{assumptionFRobuste}  holds. Let for all $j \in \{1,\dots,k\}$, $W_j$ be a  linear subspace of $\L^2 (\XX,\nu_n)$ with finite dimension and $Z_j$ be a bounded subset of $W_j$. Let then $\boldsymbol{\rho} \in [0,+\infty)^k$ such that for all $j \in \{1,\dots, k \}$,  
$ Z_j \subset \B_{\x} (0,\rho_j) = \{g \in \L^2 (\XX,\nu_n) , \, \|g\|_{\x} \leq \rho_j \} $.
Let $m_1,\dots,m_k \in \R \cup \{-\infty\}$ and $M_1,\dots,M_k \in \R \cup \{\infty\}$  be such that
 $$\Theta = \left\{ \boldsymbol{x} \in \R^k, \, \forall i \in \{1,\dots,k\}, \, m_i \leq x_i \leq M_i \right\}$$
and let $\pi$ be the map defined on $\R^k$ by
$$\pi (\boldsymbol{x}) = \big(  (x_1 \vee m_1) \wedge M_1,\dots,   (x_k \vee m_k) \wedge M_k  \big) \quad \text{for all $\boldsymbol{x} = (x_1,\dots,x_k) \in \R^k$.}$$
Let for all $\boldsymbol{u} \in   \prod_{j=1}^k Z_j$, $g_{\boldsymbol{u}}$ be the function defined by
$$g_{\boldsymbol{u} (x)} (t) = f_{\pi  \left(\boldsymbol{u} (x)\right)} (t) \quad \text{for all $(t,x) \in \TT \times \XX$.}$$
Then, the set $V$ defined by
 $$V = \left\{g_{\boldsymbol{u}}, \,  \boldsymbol{u} \in   \prod_{j=1}^k Z_j \right\}$$
has a metric dimension bounded by
$$D_V (\eta) = \frac{1}{2} \vee \frac{1}{4} \sum_{j=1}^k \log \left(1 + 2   \left(\frac{  k R_{j} }{ \eta}\right)^{1/\alpha_j}  \rho_j \right) \dim (W_j) .$$
\end{prop}

\begin{proof}[Proof of Proposition~\ref{propdimensionmetricrobuste}]
As in the proof of Lemma~\ref{lemmeproduitlinear}, we say that a set $S(\eta)$ is a $\eta$-net of a set $V$ in a metric space $(E,d)$ if, for all $y \in V$, there exists $x \in S(\eta)$ such that $d(x,y) \leq \eta$.
 
Let  $\eta > 0$ and  for $j \in \{1,\dots , k\}$, 
 $$\eta_j =  \left(\frac{\eta}{ k R_j} \right)^{1/ \alpha_j}.$$
Let $Z_j' (\eta_j)$ be a maximal subset of $Z_j$ such that $d_{\x} (x, y) > \eta_j$ for all $x \neq y \in Z_j' (\eta)$. This is a $\eta_j$-net of $Z_j$ such that
\begin{eqnarray*}
|Z_j' (\eta_j)| \leq \left|Z_j' (\eta_j) \cap \B_{\x} (0, \rho_j) \right| 
\end{eqnarray*}
and by using Lemma~4 of \cite{BirgeTEstimateurs},
\begin{eqnarray} \label{eqMajorationZjprime}
|Z_j' (\eta_j)| \leq \left(\frac{2 \rho_j}{\eta_j} + 1 \right)^{\dim W_j}.
\end{eqnarray}
To prove the proposition, we begin to show that the set
$$S(\eta) = \left\{g_{\boldsymbol{u}} , \,  \boldsymbol{u} \in \prod_{j=1}^k Z_j' \left( \eta_j \right)   \right\} $$
is a $\eta$-net of $V$.  

Let $f \in V$ be the function of the form $f(t,x) = g_{\boldsymbol{u} (x)}(t) = f_{\pi \left(\boldsymbol{u} (x)\right)} (t)$  and for all $j \in \{1,\dots,k\}$, let $v_j \in Z_j'  \big( \eta_j \big) $ such that $d_{\x} (u_j, v_j) \leq\eta_j$. We define $\boldsymbol{v} = (v_1, \dots, v_k)$ and $g \in S(\eta)$ by $g (t,x) =  g_{\boldsymbol{v} (x)}(t) = f_{\pi \left(\boldsymbol{v} (x)\right)} (t)$. Then,
\begin{eqnarray*}
\left\|f - g \right\|^2_{2} &=& 
\frac{1}{n} \sum_{i=1}^n \left\|f_{\pi \left(\boldsymbol{u} (x_i)\right)} ( \cdot) - f_{\pi \left(\boldsymbol{v} (x_i)\right)} (\cdot) \right\|^2_{\t} \\
&\leq& \frac{1}{n} \sum_{i=1}^n  \left(\sum_{j=1}^k R_j \big|u_j(x_i) - v_j(x_i) \big|^{\alpha_j} \right)^2 .
\end{eqnarray*}
By using the Cauchy-Schwarz inequality,
\begin{eqnarray*}
\left\|f - g \right\|^2_{2}  &\leq& \frac{1}{n} \sum_{i=1}^n  k \left(\sum_{j=1}^k R_j^2 \big|u_j(x_i) - v_j(x_i) \big|^{2 \alpha_j} \right) \\
&\leq&   k  \sum_{j=1}^k R_j^2   \left(\frac{1}{n} \sum_{i=1}^n \big|u_j(x_i) - v_j(x_i) \big|^{2 \alpha_j} \right).
\end{eqnarray*}
By using the concavity of the map $x \mapsto x^{\alpha_j}$,
\begin{eqnarray*}
\left\|f - g \right\|^2_{2}  &\leq&   k  \sum_{j=1}^k R_j^2   \left(\frac{1}{n} \sum_{i=1}^n \big|u_j(x_i) - v_j(x_i) \big|^{2} \right)^{\alpha_j} \\
&\leq&   k  \sum_{j=1}^k R_j^2 d_{\x}^{2 \alpha_j} (u_j,v_j) \\
&\leq& \eta^2
\end{eqnarray*}
as wished.
 
 We now consider  $x \geq 2$, $\varphi \in \L^2 (\TT \times \XX, M)$ and aim at bounding from above the cardinality of $S (\eta) \cap \B(\varphi, x \eta)$. We have,
\begin{eqnarray*}
|S (\eta) \cap \B(\varphi, x \eta)| \leq \prod_{j=1}^k \left| Z_j' \left( \eta_j \right)\right|. 
\end{eqnarray*}
By using (\ref{eqMajorationZjprime}),
\begin{eqnarray*}
|S (\eta) \cap \B(\varphi, x \eta)| &\leq& \prod_{j=1}^k \left(2  \left(\frac{  k  R_{j} }{\eta}\right)^{1/\alpha_j}  \rho_j + 1 \right)^{\dim (W_j)} \\
&\leq& \exp \left(\frac{1}{4}\sum_{j=1}^k  \dim (W_j) \log \left(2 \left(\frac{  k  R_{j} }{\eta}\right)^{1/\alpha_j}  \rho_j + 1 \right)  x^2\right).
\end{eqnarray*}
This ends the proof.
\end{proof}
\begin{lemme} \label{calculdeeta}
Let $V$ be a set with metric dimension bounded by $D_V$. Assume that there exist $k \in \N^{\star}$, $\boldsymbol{a}, \boldsymbol{b} \in [0,+\infty)^k $ such that $\max_{1 \leq j \leq k} a_j \geq 1$,  $\min_{1 \leq j \leq k} b_j  \geq 1$ and such that  
$$ D_V (\eta) \leq \frac{1}{2} \vee \sum_{j=1}^k  a_j \log \left(1+  \frac{b_j}{\eta} \right) \quad \text{for all $\eta > 0$.} $$
Then, there exists an universal positive constant $C$  such that
$$C \eta_V^2 \leq \frac{ \sum_{j=1}^k a_j \log (1+b_j)}{ n} + \frac{\sum_{j=1}^k a_j}{ n} \log \left( 1+ \frac{ n}{\sum_{j=1}^k a_j} \right)$$
where
$$\eta_V = \inf \left\{ \eta > 0 , \; \frac{D_V (\eta)}{\eta^2} \leq   n \right\}. $$
\end{lemme}

\begin{proof} [Proof of Lemma \ref{calculdeeta}]
The larger $D_V$, the larger  $\eta_V$. Consequently, 
without lost of generality we can assume that
\begin{equation*}
  D_V (\eta) =
  \begin{cases}
  2  \sum_{j=1}^k  a_j \log (2 b_j) + 2 \left( \sum_{j=1}^k a_j \right) \log \left( \frac{1}{\eta} \right) & \text{if $\eta < 1$ } \\
    2 \sum_{j=1}^k a_j \log (2 b_j) & \text{otherwise.}
  \end{cases}
\end{equation*}
Remark that for all $\alpha, \beta, y > 0$, the equation 
$$\alpha + \beta \log x = \frac{y}{2 x^2} $$
has only one positive solution $x$ given by
$$x^2 = \frac{y}{\beta  L \left( \frac{e^{\frac{2 \alpha}{\beta}}}{\beta} y \right)}, $$
where $L$ is the Lambert function, defined as being the inverse function of $t \mapsto t e^{t}$.
Consequently, by setting $$\alpha = \sum_{j=1}^k a_j \log (2 b_j) \quad \text{and} \quad \beta = \sum_{j=1}^k a_j$$ we derive that the positive number $\eta$ defined by
\begin{equation*}
\eta^2 =
  \begin{cases}
  	\frac{\beta  } {n} L \left( \frac{e^{\frac{2 \alpha}{\beta}}}{\beta} n \right) & \text{if $n > 2 \alpha$ } \\
    \frac{2 \alpha}{ n} & \text{if $n \leq 2 \alpha$  } 
  \end{cases}
\end{equation*}
is such that $D_V (\eta) = n \eta^2$. In particular, $\eta_V \leq \eta$. The conclusion ensues from some elementary inequalities on the Lambert function.
\end{proof}

We derive from this lemma the following result.
\begin{lemme} \label{lemmeMajorationetaV}
Under the notations and assumptions of  Proposition~\ref{propdimensionmetricrobuste}, there exists an universal positive constant $C$ such that
$$C \eta_{ V}^2  \leq \frac{1}{n} \sum_{j=1}^k \left( \frac{\dim (W_j) \vee 1}{\alpha_j}  \right) \left(\log \left(1 +  k R_{j} \rho_j^{\alpha_j} \right) + \log n  \right). $$
\end{lemme}
\begin{proof}
We can upper bound $D_V(\eta)$ as follows.
\begin{eqnarray*}
D_V (\eta) &=& \frac{1}{2} \vee \frac{1}{4} \sum_{j=1}^k \log \left(1 + 2 \left(\frac{  k R_{j} }{ \eta}\right)^{1/\alpha_j}  \rho_j \right) \dim (W_j)  \\
&\leq& \frac{1}{2} \vee \frac{1}{4} \sum_{j=1}^k \frac{1}{\alpha_j} \log \left( 1 + 2^{\alpha_j}  \frac{  k R_{j} }{ \eta}  \rho_j^{\alpha_j} \right) \dim (W_j).
\end{eqnarray*}
We then use Lemma~\ref{calculdeeta} with $a_j = \alpha_j^{-1}  \big(1 \vee \dim W_j\big)$ and $b_j = 1 \vee \big(2^{\alpha_j}  k R_j \rho_j^{\alpha_j}\big)$ (we recall that $\alpha_j \leq 1$). There exists thus an universal constant $C'$ such that
$$C' \eta_{ V}^2  \leq \frac{1}{n} \sum_{j=1}^k \left( \frac{\dim (W_j) \vee 1}{\alpha_j}  \right) \left[\log \left(1 +   1 \vee \big(2^{\alpha_j}   k R_j \rho_j^{\alpha_j}\big) \right) + \log \left( 1 +  \frac{n}{\sum_{i=1}^k \left( \frac{\dim (W_i) \vee 1}{\alpha_i}  \right)}\right)  \right]. $$
We now roughly upper-bound  the right-hand side of this inequality to end the proof.
\end{proof}
Let us return to the proof of Theorem~\ref{thmestimationrobuste}.
For all $W_j \in \W_j$, $\rho_j \in \N^{\star}$, we introduce the set $Z_j (W_j,\rho_j) =   W_j \cap \B_{\x} (0, \rho_j)$ where
$\B_{\x} (0, \rho_j)$ is the closed ball centered at $0$ with radius $\rho_j$ of the metric space $(\LL^2(\XX,\nu_n), d_{\x})$.
For all $\boldsymbol{W} = (W_1,\dots,W_k) \in \prod_{j=1}^k \W_j$ , $\boldsymbol{\rho} = (\rho_1,\dots, \rho_k) \in \left(\N^{\star} \right)^k$, we define
$$V \left(\boldsymbol{W}, \boldsymbol{\rho} \right)  = \left\{(t,x) \mapsto f_{\pi \left(\boldsymbol{u} (x)\right)} (t) , \, \boldsymbol{u} \in \prod_{j=1}^k Z_j (W_j, \rho_j)  \right\}.$$
We then define  $$\V = \left\{V \left(\boldsymbol{W}, \boldsymbol{\rho} \right)  ,\,  \boldsymbol{W} \in \prod_{j=1}^k \W_j , \, \boldsymbol{\rho} \in \left(\N^{\star} \right)^k  \right\}$$
and we define the map $\Delta$ on $\V$ by
$$\Delta \left( V \left(\boldsymbol{W},\boldsymbol{\rho} \right) \right) = \sum_{j=1}^k \left( \Delta_j (W_j) +   \log \left(2 \rho_j^2 \right) \right). $$
We apply Theorem~\ref{thmselectionmodelegeneral} with $(\V,\Delta)$ to build an estimator~$\hat{s}$. For all $\boldsymbol{W} = (W_1,\dots,W_k) \in \prod_{j=1}^k \W_j$ ,  $\boldsymbol{\rho} = (\rho_1,\dots,\rho_k) \in \left(\N^{\star} \right)^k$,   
\begin{eqnarray} \label{eqDansPreuveetaV}
C \E\left[H^2(s, \hat{s}) \right]  \leq    d^2_2 \left(\sqrt{s},   V \left(\boldsymbol{W}, \boldsymbol{\rho} \right)\right) + \eta_{ V \left(\boldsymbol{W}, \boldsymbol{\rho} \right)}^2 + \frac{  \Delta \left( V\left(\boldsymbol{W}, \boldsymbol{\rho} \right) \right) }{n}    
\end{eqnarray}
where $C$ is an universal positive constant.
We then derive from Lemma~\ref{lemmeMajorationetaV} that there exists an universal positive constant $C'$ such that
$$C' \eta_{ V \left(\boldsymbol{W}, \boldsymbol{\rho} \right)}^2  \leq \frac{1}{n} \sum_{j=1}^k \left( \frac{\dim (W_j) \vee 1}{\alpha_j}  \right) \left(\log \left(1 +  k R_{j} \rho_j^{\alpha_j} \right) + \log n  \right). $$
In particular, for all function $f \in \F$ of the form $f (t,x) = f_{\boldsymbol{u}(x)}(t)$, for all $\boldsymbol{W}  \in \prod_{j=1}^k \W_j$, for all map $\boldsymbol{v} = (v_1,\dots,v_k) \in \prod_{j=1}^k W_j$, such that for all $j \in \{1,\dots,k\}$,  $\|v_j\|_{\x} \leq \|u_j\|_{\x}$, and for all function~$g$ of the form $g(t,x) = f_{\pi \left(\boldsymbol{v}(x) \right)} (t)$,   inequality (\ref{eqDansPreuveetaV}) used with $\rho_j = \inf \{i \in \N^{\star} , \, i \geq  \|u_j\|_{\x} \}$ leads to
\begin{eqnarray*}
 C'' \E\left[H^2(s, \hat{s}) \right] &\leq&   d^2_2 \left(\sqrt{s},  f  \right) +  d^2_2 \left( f, g  \right)+ \frac{ \sum_{j=1}^k \left( \Delta_j (W_j) +  \log (1+ \|u_j\|_{\x}) \right)  }{n}     \\
& &  +  \frac{1}{n} \sum_{j=1}^k \left( \frac{\dim (W_j) \vee 1}{\alpha_j}  \right) \left[\log \left(1 +  k R_{j}  (1 + \|u_j\|_{\x})^{\alpha_j } \right) + \log n  \right]
\end{eqnarray*}
where $C''$ is an universal positive constant.

By using Assumption~\ref{assumptionFRobuste} and the Cauchy-Schwarz inequality,
\begin{eqnarray*}
 d^2_2 \left( f , g  \right)  \leq k \sum_{j=1}^k R_{j}^2 \|u_j  - v_j  \|_{\x}^{2 \alpha_j}.
 \end{eqnarray*}
We then  choose $v_j$ as being the projection of $u_j$ on $W_j$ in the space $\L^2 (\XX,\nu_n)$, and take the infimum over all $\boldsymbol{W}  \in \prod_{j=1}^k \W_j$ to conclude.

\subsection{Proof of Corollary~\ref{corrobustepowerlaw}.}
Let $r_1,r_2 \in \N^{\star}$. We may apply Theorem~\ref{thmestimationrobuste} with collections $\W_1, \W_2$ provided by Proposition~1 of~\cite{BaraudComposite}. This yields an estimator~$\tilde{s}$ such that for all  $a \in \mathcal{H}^{\boldsymbol{\alpha}} ([0,1]^{k_2})$ with values into $[-r_1,r_1]$, for all $b \in \mathcal{H}^{\boldsymbol{\beta}} ([0,1]^{k_2})$ with values into $[-1/2+1/r_2,+\infty)$, 
\begin{eqnarray*} \label{eqestmiationrobuste}
C \E\left[ H^2(s,\tilde{s})  \right]  \leq    d_2^2 ( \sqrt{s} , f ) +  \varepsilon_1 (a)  +  \varepsilon_2 (b)  
\end{eqnarray*}
where 
\begin{eqnarray*}
C_1 \varepsilon_1 (a) &\leq& \left(r_2^{1/2} L(a) \right)^{\frac{2 k_2}{k_2 + 2  \bar{\boldsymbol{\alpha}}}} \left(\frac{ \log  n  + \log (1 \vee r_2^{1/2}) + \log \left( 1 \vee \|a\|_{\x} \right)}{ n} \right)^{\frac{2  \bar{\boldsymbol{\alpha}}}{2 \bar{\boldsymbol{\alpha}} + k_2}}  \\
& & \quad   + \frac{\log  n  + \log (1 \vee r_2^{1/2}) + \log \left( 1 \vee \|a\|_{\x} \right) }{ n}  \\
C_2 \varepsilon_2 (b) &\leq& \left(\sqrt{2} r_1 r_2^{3/2} L(b) \right)^{\frac{2 k_2}{k_2 + 2  \bar{\boldsymbol{\beta}}}} \left(\frac{ \log  n  + \log (1 \vee \sqrt{2} r_1 r_2^{3/2}) + \log \left( 1 \vee \|b\|_{\x} \right)}{ n} \right)^{\frac{2  \bar{\boldsymbol{\beta}}}{2 \bar{\boldsymbol{\beta}} + k_2}}  \\
& & \quad   + \frac{\log  n  + \log (1 \vee \sqrt{2} r_1 r_2^{3/2}) + \log \left( 1 \vee \|b\|_{\x} \right) }{ n}  
\end{eqnarray*}
where  $C > 0$ is universal, where $C_1 > 0$ depends only on $k_2$, $\max_{1 \leq j \leq k_2} \alpha_j$, and where $C_2 > 0$ depends only on $k_2$, $\max_{1 \leq j \leq k_2} \beta_j$.

The above estimator depends on $r_1$ and $r_2$. We can thus use Proposition~\ref{propmixing}, to derive that there exists an estimator $\hat{s}$, such that for all $r_1,r_2 \in \N^{\star}$, for all  $a \in \mathcal{H}^{\boldsymbol{\alpha}} ([0,1]^{k_2})$ with values into $[-r_1,r_1]$, for all $b \in \mathcal{H}^{\boldsymbol{\beta}} ([0,1]^{k_2})$ with values into $[-1/2+1/r_2,+\infty)$, 
\begin{eqnarray*} \label{eqestmiationrobuste}
C'' \E\left[ H^2(s,\hat{s})  \right]  \leq    d_2^2 ( \sqrt{s} , f ) +  \varepsilon_1 (a)  +  \varepsilon_2 (b)  + \frac{\log r_1 + \log r_2}{n}
\end{eqnarray*}
where $C'' > 0$ is universal.

In particular, if we choose $r_1$ as being the smallest integer larger than $\|a\|_{\infty}$ and $r_2$ as being the smallest integer larger than $2/(\inf_{x \in [0,1]^{k_2}} (2 b(x) + 1))$, we get
\begin{eqnarray*}
C_1' \varepsilon_1 (a) \leq \left( \sqrt{1 + \frac{2}{\inf_{x \in [0,1]^{k_2}} (2 b(x) + 1)} } L(a) \right)^{\frac{2 k_2}{k_2 + 2  \bar{\boldsymbol{\alpha}}}} \left(\frac{ \log  n}{n} \right)^{\frac{2  \bar{\boldsymbol{\alpha}}}{2 \bar{\boldsymbol{\alpha}} + k_2}}  + C''_1 \frac{\log n}{n}
\end{eqnarray*}
where $C_1'  > 0$ depends only on $k_2$, $\max_{1 \leq j \leq k_2} \alpha_j$ and where $C''_1$ depends only on $k_2, \bar{\boldsymbol{\alpha}}, \|a\|_{\infty}, L(a), \|b\|_{\infty}$ and $\inf_{x \in [0,1]^{k_2}} (2 b(x) + 1)$.
We then use 
$$1 + \frac{2}{\inf_{x \in [0,1]^{k_2}} (2 b(x) + 1)} \leq \frac{3}{1 \wedge \inf_{x \in [0,1]^{k_2}} (2 b(x) + 1)}$$
to get
\begin{eqnarray*}
 C_1''' \varepsilon_1 (a) \leq \left(\frac{1}{1 \wedge \inf_{x \in [0,1]^{k_2}} (2 b(x) + 1)} \right)^{\frac{ k_2}{k_2 + 2  \bar{\boldsymbol{\alpha}}}}  L(a)^{\frac{2 k_2}{k_2 + 2  \bar{\boldsymbol{\alpha}}}} \left(\frac{ \log  n}{n} \right)^{\frac{2  \bar{\boldsymbol{\alpha}}}{2 \bar{\boldsymbol{\alpha}} + k_2}}  + C_1'''' \frac{\log n}{n}.
\end{eqnarray*}
We bound from above $ \varepsilon_2 (b)$ in the same manner. \qed

\section*{Acknowledgement}
Many thanks to Yannick Baraud for his suggestions, comments and careful reading of the paper.
\bibliographystyle{apalike}
 \bibliography{bibliopubcox}

\end{document}